\documentclass{svjour3}

\usepackage{amsmath,amsfonts,eucal,pstricks,tikz,pgfkeys}

\usepackage{subfig,graphicx}

\def\fr#1/#2.{\frac{#1}{#2}}
\def\tr{^{\mbox{\tiny\sf\bfseries T}}}
\def\H{^{\mbox{\tiny\sf\bfseries H}}}
\DeclareMathOperator{\diag}{diag}
\def\R{{\mathbb R}}
\def\N{{\mathbb N}}
\def\Z{{\mathbb Z}}
\def\C{{\mathbb C}}
\def\M{{\cal M}}
\def\S{{\cal S}}
\def\F{{\cal F}}
\def\Y{{\cal Y}}

\def\1{\mbox{\bf 1}}
\def\0{\mbox{\bf 0}}

\newtheorem{prop}[theorem]{Proposition}
\newtheorem{defn}[theorem]{Definition}

\def\B{\mbox{\small \sf\bfseries B}}
\def\ta{\psline{*-}(0,0)(0,0)}
\def\tb{\psline{*-*}(0,0)(0,.3)}
\def\tc{\psline{*-*}(0,0)(.2,.3)\psline{-*}(0,0)(-.2,.3)}
\def\td{\psline{*-*}(0,0)(0,.3)\psline{-*}(0,0.3)(0,.6)}
\def\te{\psline{*-*}(0,0)(.25,.3)\psline{-*}(0,0)(0,.3)
\psline{-*}(0,0)(-.25,.3)}
\def\tf{\psline{*-*}(0,0)(.2,.3)\psline{-*}(0,0)(-.2,.3)
\psline{-*}(0.2,.3)(.2,.6)}
\def\tg{\psline{*-*}(0,0)(0,.3)\psline{-*}(0,0.3)(.2,.6)\psline{-*}(0,0.3)(-.2,.6)}
\def\th{\psline{*-*}(0,0)(0,.3)\psline{-*}(0,0.3)(0,.6)
\psline{-*}(0,0.6)(0,.9)}

\def\ii{\mbox{\bf i}}
\DeclareMathOperator{\re}{Re}
\DeclareMathOperator{\im}{Im}
\def\m{\phantom{-}}

\begin{document}


\newpage
\title{Symmetric general linear  methods}
\author{J.~C.~Butcher \and A.~T.~Hill \and T.~J.~T.~Norton}


\institute{J.~C.~Butcher \at
            Department of Mathematics,
            University of Auckland, NZ \\
              \email{butcher@math.auckland.ac.nz}           
           \and
           A.~T.~Hill  \at
              Department of Mathematical Sciences,
               University of Bath, UK\\
              \email{masath@bath.ac.uk}        
             \and
             T.~J.~T.~Norton \at
              Department of Mathematical Sciences,
               University of Bath, UK\\
              \email{tjtn20@bath.ac.uk}             
}

\date{Received: date / Accepted: date}
\maketitle

\begin{abstract}
The article considers symmetric general linear methods, a class of numerical time integration methods which, like symmetric Runge--Kutta methods, are  applicable to general time--reversible differential equations, not just those derived from separable second--order problems. A definition of time--reversal symmetry is formulated for general linear methods, and criteria are found for the methods to be free of linear parasitism. It is shown that   symmetric parasitism--free methods   cannot be explicit, but  a method of order $4$ is constructed with only one implicit stage. Several characterizations of symmetry are given, and  connections are made with $G$--symplecticity. Symmetric methods are shown to be of even order, a suitable symmetric starting method  is  constructed and shown to be essentially unique. The underlying one--step method is shown to be time--symmetric. Several symmetric methods of order $4$ are constructed and implemented on test problems. The methods are efficient when compared with Runge--Kutta methods of the same order, and invariants of the motion are well--approximated over long time intervals.
\end{abstract}

\section{Introduction}\label{sec:intro}
Symmetric general linear methods are a class of multistage multivalue methods with time--reversal symmetry. As we demonstrate, such methods can efficiently integrate the solutions of differential equations which are themselves time--reversible, in such a way that invariants of the motion are preserved over long time intervals. The main aim of this paper is to characterize, construct and test high--order symmetric general linear methods with minimal implicitness and zero parasitic growth--parameters.

Under mild conditions, the flow  associated with a general ordinary differential equation  satisfies the basic time--reversal symmetry $E_{-x}E_x =I$. A Runge--Kutta method is symmetric if it satisfies the analogous property,
\begin{equation}
\M_{-h}\M_{h}=I,\label{eq:basicsymm}
\end{equation}
where $\M_h$ is the map generated by a single step of the method. As shown in \cite{st88}, \cite{hs97}, \eqref{eq:basicsymm} is also sufficient for a Runge--Kutta method to inherit the stronger symmetry of $\rho$--reversibility, when a differential equation has this property. Symmetry implies even order and  leads to  simplifications in the order theory for such methods, \cite{mss99}. Practically, symmetric Runge--Kutta methods are  shown  to perform well for such problems over long time intervals in the book of Hairer, Lubich \& Wanner \cite{hlw}.  However, every irreducible  stage of a symmetric Runge--Kutta method is necessarily implicit, \cite{ste73}, \cite{wan73},\cite[V.2]{hlw}. The most efficient such methods are DIRKs, formed by compositions of the implicit midpoint method, \cite{ssab}, \cite{y90}, \cite{suz90}, \cite{mcl95}. For separable problems originating from a system of second order differential equations, the symplectic Euler and Runge--Kutta--Nystr\"{o}m methods have been generalized to obtain higher order  partitioned Runge--Kutta methods \cite{hlw}, some of which are explicit. The most popular low order method for separable problems is the explicit St\"{o}rmer--Verlet method \cite{ver67}, which may be viewed  as a partitioned Runge--Kutta method, a partitioned linear multistep method, or a non--standard implementation of the leapfrog method.

The properties of standard linear multistep methods and one--leg methods were investigated by Eirola \& Sanz--Serna \cite{es92}, who showed that symmetry is equivalent to $G$--symplecticity in this case. The properties of symmetric multistep methods were further investigated in \cite{css98}. However, Dahlquist \cite{d56} had already shown that the parasitic roots of such methods have non--zero growth--parameters. Hence, symmetric linear multistep and one--leg methods are weakly unstable.

An important class of systems with time--reversal symmetry are of the form
\begin{equation*}
\frac{d^{2}y}{dx^{2}}=f(y),
\end{equation*}
familiar from many examples in Mechanics and other branches of Physics. 
 The classical St\"{o}rmer--Cowell linear multistep methods  \cite{st07}, \cite{cow10}, popular with Astronomers, exploit the special structure of such systems by directly approximating the second derivative. However, only lowest order   method (St\"{o}rmer--Verlet) is symmetric. The articles \cite{d56} and \cite{lw76} made early studies of the  stability properties of  second order multistep methods. New symmetric high order second order methods  were designed and successfully tested in \cite{qt90}. Hairer \& Lubich  \cite{hl04}, \cite{hlw2} used backward error analysis techniques to show that the underlying one--step method is a symmetric approximation of the true solution, and that parasitic components remain under control for long times. 

As a model for  general linear methods, consider  a $k$--step linear multistep method in one--step form. Here,  $\M_h$ may be interpreted as the map 
\begin{equation}\label{eq:symmmult}
\begin{split}
&[y_{n},\,\ldots,\,y_{n+k-1},\,hf_{n},\,\ldots,\,hf_{n+k-1}]\\
&\hspace{3cm} \mapsto  [y_{n+1},\,\ldots,\,y_{n+k},\,hf_{n+1},\,\ldots,\,hf_{n+k}],
\end{split}
\qquad n\in \Z.
\end{equation}
Under the change of variable $m=n+k$, $\M_{-h}^{-1}$ represents the mapping 
\begin{equation*}
\begin{split}
&[y_{m+k-1},\,\ldots,\,y_{m},\,-hf_{m+k-1},\,\ldots,\,-hf_{m}]\\
&\hspace{3cm}\mapsto  [y_{m+k},\,\ldots,\,y_{m+1},\,-hf_{m+k},\,\ldots,\,-hf_{m+1}],
\end{split} \qquad m\in \Z.
\end{equation*}
To equate this mapping with \eqref{eq:symmmult}, a coordinate transform $L$ is needed,  which reverses the order of both the $y$ and $hf$ entries. Furthermore, $L$ must multiply the $hf$ terms by $-1$. (Both these actions of L  are directly related to time--reversal.) Then, the following modification of \eqref{eq:basicsymm} holds:
\begin{equation}
L\M_{-h}L\M_{h}=I,\qquad \mbox{ such that }\quad L^2=I.\label{eq:glmsymm}
\end{equation}
As shown in Section \ref{sec:symm}, identity \eqref{eq:glmsymm} characterizes a symmetric general linear method.  In Section \ref{sec:even}, it is shown that \eqref{eq:glmsymm} implies that \eqref{eq:basicsymm} is formally satisfied by the corresponding underlying one--step method. These results are essentially similar to those of \cite[XIV.4.2]{hlw}, though our assumptions differ in detail.

Three further characterizations of symmetry are obtained in the paper:\\
(i) In Section \ref{sec:symm}, an algebraic condition \eqref{eq:symm} in terms of the method coefficient matrices $(A,\,U,\,B,\,V)$, the matrix $L$  and a stage permutation matrix $P$, which also satisfies $P^2=I$; see also \cite[XIV.4.2]{hlw}. This condition, together with the canonical form identified later in Section \ref{sec:even}, is the most useful in method construction.\\
(ii) In Section \ref{sec:stab}, an $AN$--stability condition: $LM(-PZP)LM(Z)=I$ for all sufficiently small diagonal $Z$, where  $M(Z)$ is the non--autonomous linear stability matrix, cf. \cite{b87}. This condition helps to show linear stability on a subinterval of the imaginary axis.\\
(iii) Also in Section \ref{sec:stab}, a characterization in terms of the matrix transfer function, generalizing the one--leg condition of \cite{es92}, $(\sigma/\rho)(\zeta)=-(\sigma/\rho)(\zeta^{-1})$. This condition has potential application to long--time nonlinear stability theory, and also helps in the construction of methods that are both symmetric and G--symplectic.
 
Parasitism is a potential disadvantage for any non--trivial symmetric general linear methods. However, in Section \ref{sec:stab}, we find necessary and sufficient conditions on the coefficient matrices of the method for the linear stability matrix $M(zI)$ to have sublinear growth in parasitic directions. (In the terminology of \cite[XIV.5.2]{hlw}, this is equivalent to all parasitic roots having zero  growth--parameters.) These coefficient conditions play a critical role in the construction of practical methods in Section \ref{sec:examp}. They are also used to show that there are no explicit symmetric parasitism--free methods.

In Section \ref{sec:even}, it is  shown that a symmetric general linear method is always of even order. Central to this result is a constructive proof of the existence and uniqueness of a symmetry--respecting starting method $\S_h$ satisfying 
\begin{equation}
\S_h=L\S_{-h},\qquad \label{eq:symmstart}
\end{equation} 
with respect to which $\M_h$ is of maximal order. Related ideas are used to show the existence of a formal starting method, underlying one--step method pair $(\S_h,\,\Phi_h)$ such that
\begin{equation}
\M_h\S_h=\S_h\Phi_h.\qquad \label{eq:uosmi}
\end{equation}

Example symmetric methods of order $4$ are constructed in Section \ref{sec:examp}. These methods have diagonally implicit stage matrices, and some are also $G$--symplectic. One method has only one implicit stage, and is therefore theoretically more efficient than a symmetric DIRK of the same order. The simulations of Section \ref{sec:simul} show that symmetric general linear methods approximately conserve the Hamiltonian of several low--dimensional symmetric problems over long time intervals in a similar way to symmetric Runge--Kutta methods. Furthermore, there are $4$th order symmetric general linear methods with fewer implicit stages than is possible in the Runge--Kutta case. This leads to some efficiency savings over long--times.


\section{General linear methods}\label{sec:GLM}
For $X=\R^N$, $f: X \to X,$ and $y_0\in X$, let $y=y_{y_0}$ denote the solution of  the autonomous initial value problem,
\begin{equation}
y'(x)=f(y(x)),\qquad  y(0)=y_0.\label{eq:ODE}
\end{equation}
For $x\in \R$, denote the flow for \eqref{eq:ODE} by  $E_{x}: X\longrightarrow X$, so that 
\begin{equation*}
y(x)=E_x y_0.
\end{equation*} 
For all ODEs, the evolution operator satisfies the group properties,
\begin{equation}
E_0=I;\qquad E_{x_1}E_{x_2}=E_{x_{1}+x_2},\quad x_1,\; x_2\in \R; \qquad E_x^{-1} = E_{-x},\quad x\in \R\label{eq:group}
\end{equation}

We refer to a general linear method $(A,\,U,\,B,\,V)$, where 
\begin{equation}\label{eq:method}
\left[
\begin{array}{cc} A&U\\B&V   \end{array}
\right]
\end{equation}
forms  a partitioned $(s+r)\times(s+r)$ complex--valued matrix or tableau. For practical methods, the coefficients are real, but for some theoretical purposes   the complex case is also treated.

For  time--step $h$ and $n\in \N$,  the new values $y^{[n]}\in X^r$ are  found from  $y^{[n-1]}\in X^r$ via the formulae
\begin{align}
Y &= h(A\otimes I) F + (U\otimes I)  y^{[n-1]}, \label{eq:Y}\\
y^{[n]} &= h(B\otimes I) F + (V\otimes I)  y^{[n-1]},\label{eq:y1}
\end{align}
 defined using temporary $Y,F\in X^s$. The subvectors in $F$ (the stage derivatives) are related to the subvectors in $Y$ (the stages) by $F_i = f(Y_i)$, $i=1,2,\dots,s$. Usually, where no ambiguity is possible, the Kronecker products in \eqref{eq:Y} and \eqref{eq:y1} will be omitted
and we  write  
\begin{align*}
Y &= hAF + Uy^{[n-1]},\\
y^{[n]} &= hB F + Vy^{[n-1]}.
\end{align*}

In this paper, the method is always assumed to satisfy the conditions below.
\begin{defn}
\label{defn:cons}A general linear method $(A,U,B,V)$ is 
\begin{itemize}
\item[\rm(\ref{defn:cons}a)]  {Preconsistent}, if  $(1,\,u,\,w\H)$ is an eigentriple of $V$, such that
\begin{equation*}
Vu=u,\qquad w\H V=w\H,\qquad w\H u=1;
\end{equation*}
\item[\rm(\ref{defn:cons}b)]  {Consistent}, if it is preconsistent,  $Uu=\1$, and there exists non--zero $v\in \C^r$   such that $B\1+ Vv=u+v$;
\item[\rm(\ref{defn:cons}c)]  {Zero-stable}, if  $\sup_{n\in{\mathbb N}_0}\|V^{n}\|<\infty$.
\end{itemize}
\end{defn}

To approximate the solution of  \eqref{eq:ODE} with initial data $y_0\in X$,  we generate
\begin{equation*}
y^{[0]}= S_{h}y_0,
\end{equation*}
using a practical starting method $S_{h}: X\longrightarrow X^r$, where the tableau 
\begin{equation}
\left[
\begin{array}{cc} \widetilde{A}&\1\\\widetilde{B}&u   \end{array}
\right],\label{eq:start}
\end{equation}
has dimensions $(\widetilde{s}+r)\times (\widetilde{s}+1)$. Similarly, a practical finishing method, $F_h: X^r\longrightarrow X$, is required. It is  assumed that $F_h\circ S_{h}=I_X$. 

\section{Symmetric methods}\label{sec:symm}
We define symmetry in the context of the nonlinear map  generated by the method.   Other characterizations of symmetry
are considered, with a view to identifying or constructing symmetric methods.

\subsection{The method as a nonlinear map}

For $f: X \to X$ and  time--step $h$,  the method maps an input vector $y\in X^r$ to an output vector $\M_h y$. Define the nonlinear map $\M_h: X^r\longrightarrow X^r$ by
\begin{align}
Y &= hAF + Uy,\label{eq:stage}\\
\M_h y &= hB F + Vy.\label{eq:Mh}
\end{align}
(It will be assumed that $f$ and $h$ are such that \eqref{eq:Y} has a solution, and that a suitable selection principle chooses a unique $Y$ when multiple solutions exist.)

\noindent
{\bf  Equivalent  maps:} The map $\M_h$ is not changed if a different ordering is chosen for the subvectors of $Y$; that is, $\M_h$ is also generated by the method defined by the tableau
\begin{equation}
\left[
\begin{array}{cc} P^{-1}AP&P^{-1}U\\BP&V   \end{array}
\right],\label{eq:perm}
\end{equation}
where $P$ is a permutation matrix.

If $T\in \C^{r\times r}$ be  non--singular, then  $T^{-1}\M_h T$ is {\bf equivalent} to $\M_h$. The identity
\begin{equation*}
(T^{-1}\M_h T)(T^{-1}y)=T^{-1}\M_h y,
\end{equation*}
shows that $T$ only changes the coordinate basis. A tableau for $T^{-1}\M_h T$ is
\begin{equation}
\left[
\begin{array}{cc} A&UT\\T^{-1}B&T^{-1}VT   \end{array}
\right].\label{eq:transf}
\end{equation}

\subsection{Symmetry of the map}
We say that the  {\bf map is symmetric} if  the process of calculating $y^{[1]}$ from  $y^{[0]}$ can be reversed by using an equivalent map with the sign of $h$ reversed; i.e. 
\begin{equation}
\M_{h}=L\M_{-h}^{-1}L,\label{eq:altsymm} 
\end{equation}
for some nonsingular matrix $L\in \C^{r\times r}$, such that $L^2=I$. Physically, the involution $L$ corresponds to a linear change of coordinates for $y$ to take account of the change in  time direction.  Algebraically, the condition $L^2=I$ is required to ensure that we recover $\M_{h}$ after two iterations of \eqref{eq:altsymm}. This definition is similar to that stated in \cite[XIV]{ hlw}.

\noindent 
{\bf The inverse map:} From \eqref{eq:stage} and \eqref{eq:Mh}, we deduce that the inverse map $\M_{h}^{-1}$ satisfies 
\begin{align*}
Y &= hA F + U \M_h^{-1} y, \\
y &= hB F + V \M_h^{-1} y.
\end{align*}
Solving these equations for $\M_h^{-1}y$ yields
\begin{align}
Y &= -h(UV^{-1}B-A)F + U V^{-1}y,\label{eq:istage}\\
\M^{-1}_h y &=- hV^{-1}B F + V^{-1} y.\label{eq:Minv}
\end{align}

\subsection{Symmetry of the method}
We say that the {\bf method is symmetric}  if
\begin{equation}
\left[
\begin{array}{cc} A+PAP-UV^{-1}B&PU-ULV\\ BP-VLB&L-VLV \end{array}
\right]=0,\qquad L^2=I,\quad P^2=I,\label{eq:symmdef}
\end{equation}
More specifically, we say that method $(A,\,U,\,B,\,V)$ is $(L,\,P)$--symmetric if \eqref{eq:symmdef} holds.
\begin{prop} 
Suppose that $\M_h$ is the map associated with a symmetric method. Then, $\M_h$ is symmetric.
\end{prop}
\begin{proof} A rearrangement of definition \eqref{eq:symmdef} yields
\begin{equation}
\left[
\begin{array}{cc} A&U\\B&V   \end{array}
\right]=\left[
\begin{array}{cc} P(UV^{-1}B-A)P&PUV^{-1}L\\LV^{-1}BP&LV^{-1}L   \end{array}
\right].\label{eq:symm}
\end{equation}
Here, the left--hand side of \eqref{eq:symm} is a tableau for $\M_h$. Taking note of \eqref{eq:perm}, \eqref{eq:altsymm}, \eqref{eq:istage} and \eqref{eq:Minv}, the right--hand side of \eqref{eq:symm} is one possible tableau for $L\M_{-h}^{-1}L$.\end{proof}

\noindent
{\bf Remark:} The tableau on the right--hand side of \ref{eq:symm} is also known as an {\bf adjoint} tableau for the method $(A,\,U,\,B,\,V)$. The conditions $L^2=I$ and $P^2=I$   ensure that the original tableau is recovered after $2$ iterations of \eqref{eq:symm}. The coefficient conditions in \eqref{eq:symm}  are similar to those given in \cite[XIV]{hlw}, except that $L$ and $P$ are not involutions there.
 
\subsection{Symmetry of the starting method}
In order to ensure that $\M_h^{n}S_h=L\M_{-h}^{-n}S_{-h}$, it is required that the starting method satisfies
\begin{equation}
S_{-h}=LS_{h}.\label{eq:startminus}
\end{equation}
Considering the   tableau \eqref{eq:start}, this is equivalent to the coefficient conditions
\begin{equation}
\widetilde{A}=-\widetilde{P}\widetilde{A}\widetilde{P}\qquad \widetilde{B}=-L\widetilde{B}\widetilde{P},\label{eq:startsymm}
\end{equation}
for some permutation matrix $\widetilde{P}\in \R^{\widetilde{s}\times\widetilde{s}}$ such that $\widetilde{P}^{2}=I$.

The diagram in Figure \ref{fig:starts} shows the relationship between various quantities and mappings which have arisen in this discussion.  In addition to  $\M_h$, we introduce a further mapping $\Y_h$ defined as $Y=\Y_h y$ in \eqref{eq:Y}.
 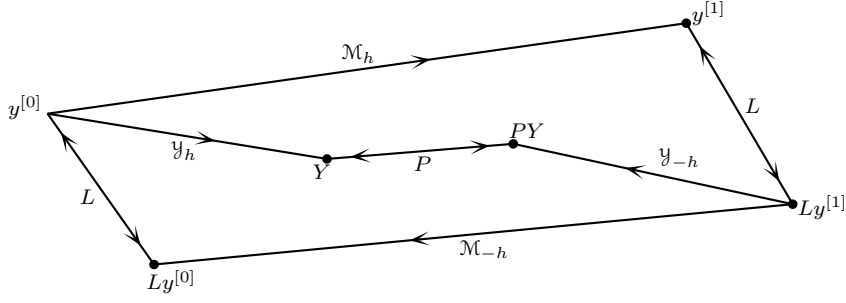
\begin{figure}
\begin{center}
\psset{xunit=3.5cm,yunit=4cm,arrowsize=2pt 2,arrowinset=0.6}
\begin{pspicture}(0,0.4)(3.6,1.6)
\psline{-*}(0.4,1.2)(2.8,1.5)
\psline{-*}(0.8,0.7)(3.2,0.9)
\psline{-*}(0.4,1.2)(1.45,1.05)
\psline{-*}(3.2,0.9)(2.15,1.1)
\psline{-*}(0.4,1.2)(0.8,0.7)
\psline(2.8,1.5)(3.2,0.9)
\psline(1.45,1.05)(2.15,1.1)
\uput{2pt}[160](0.4,1.2){$y^{[0]}$}
\uput{2pt}[-60](0.8,0.7){$Ly^{[0]}$}
\uput{2pt}[30](2.8,1.5){$y^{[1]}$}
\uput{2pt}[-10](3.2,0.9){$Ly^{[1]}$}
\uput{2pt}[-110](1.45,1.05){$Y$}
\uput{2pt}[60](2.15,1.1){$PY$}
\psset{arrowsize=3pt 3,linestyle=none}
\psline{->}(0.4,1.2)(1.84,1.38)
\psline{->}(3.2,0.9)(1.76,0.78)
\psline{->}(0.4,1.2)(1.03,1.11)
\psline{->}(3.2,0.9)(2.57,1.02)
\psline{<->}(0.45,1.1375)(0.75,0.7625)
\psline{<->}(2.85,1.425)(3.15,0.975)
\psline{<->}(1.5375,1.05625)(2.0625,1.09375)
\uput{2pt}[110](1.6,1.35){$\M_h$}
\uput{2pt}[-70](2,0.8){$\M_{-h}$}
\uput{2pt}[-150](0.6,0.95){$L$}
\uput{2pt}[30](3,1.2){$L$}
\uput{2pt}[-80](1.8,1.075){$P$}
\uput{2pt}[-100](0.925,1.125){$\Y_h$}
\uput{2pt}[90](2.775,1.0){$\Y_{-h}$}
\end{pspicture}
\end{center}
\caption{Relationships between various mappings}\label{fig:starts}
\end{figure}

In  Figure  \ref{fig:Estarts},   the role of the underlying one--step pair $(\S_h,\,\Phi_h)$, discussed in the Introduction,  is also included.  
\begin{figure}
\begin{center}
\psset{xunit=2.5cm,yunit=2.5cm,arrowsize=2pt 2,arrowinset=0.6}
\begin{pspicture}(0,0)(3.2,1.5)
\psset{arrowsize=3pt 3}
\psline{-*}(0,0)(2.5,0.25)
\psline(2.5,0.25)(2.8,1.5)
\psline(2.5,0.25)(3.2,0.9)
\pscircle[linestyle=none,fillstyle=solid,fillcolor=white](2.65,0.85){2pt}
\pscircle[linestyle=none,fillstyle=solid,fillcolor=white](2.68,1){2pt}
\psline[linestyle=none]{<->}(0.75,.075)(1.75,0.175)
\psline[linestyle=none]{->}(2.5,0.25)(2.62,0.75)
\psline[linestyle=none]{->}(2.5,0.25)(2.92,0.64)
\uput[45](0.75,.075){$\Phi_{-h}$}
\uput[140](1.75,0.175){$\Phi_{h}$}
\uput{4pt}[210](2.62,0.75){$\S_h$}
\uput{4pt}[-60](2.92,0.64){$\S_{-h}$}
\uput{5pt}[-30](2.5,0.25){$y_1$}
\psline{*-*}(0,0)(0.4,1.2)\psline{-*}(0.4,1.2)(2.8,1.5)
\psline{-*}(0,0)(0.8,0.7)\psline{-*}(0.8,0.7)(3.2,0.9)
\psline{-*}(0.4,1.2)(1.45,1.05)
\psline{-*}(3.2,0.9)(2.15,1.1)
\psline{-*}(0.4,1.2)(0.8,0.7)
\psline(2.8,1.5)(3.2,0.9)
\psline(1.45,1.05)(2.15,1.1)
\uput{4pt}[180](0,0){$y_0$}
\uput{2pt}[160](0.4,1.2){$y^{[0]}$}
\uput{2pt}[-60](0.8,0.7){$Ly^{[0]}$}
\uput{2pt}[30](2.8,1.5){$y^{[1]}$}
\uput{2pt}[-10](3.2,0.9){$Ly^{[1]}$}
\uput{2pt}[-110](1.45,1.05){$Y$}
\uput{3pt}[95](2.15,1.1){$PY$}
\uput{2pt}[160](0.2,0.6){$\S_h$}
\uput{2pt}[-40](0.4,0.35){$\S_{-h}$}
\psset{linestyle=none}
\psline{->}(0,0)(0.24,0.72)
\psline{->}(0.4,1.2)(1.84,1.38)
\psline{->}(0,0)(0.48,0.42)
\psline{->}(3.2,0.9)(1.76,0.78) 
\psline{->}(0.4,1.2)(1.03,1.11)
\psline{->}(3.2,0.9)(2.57,1.02) 
\psline{<->}(0.45,1.1375)(0.75,0.7625)
\psline{<->}(2.85,1.425)(3.15,0.975)
\psline{<->}(1.5375,1.05625)(2.0625,1.09375)
\uput{2pt}[110](1.6,1.35){$\M_h$}
\uput{2pt}[-70](2,0.8){$\M_{-h}$}
\uput{2pt}[-150](0.6,0.95){$L$}
\uput{2pt}[30](3,1.2){$L$}
\uput{2pt}[-80](1.8,1.075){$P$}

\uput{2pt}[-100](0.925,1.125){$\Y_h$}
\uput{2pt}[90](2.5,1.02){$\Y_{-h}$}
\end{pspicture}

\end{center}
\caption{The role of the underlying one-step method}\label{fig:Estarts}
\end{figure}
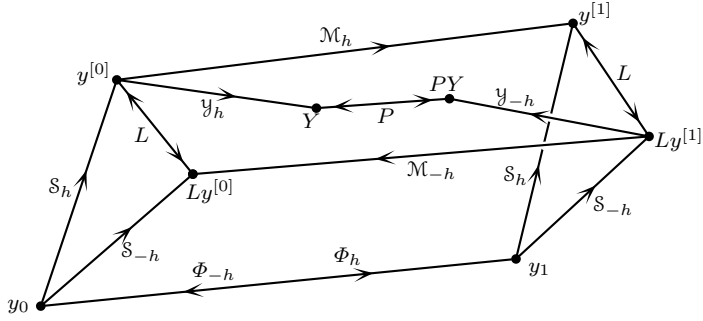


\subsection{Canonical form based on $V$-diagonalization}\label{subs:V}
Given a stable consistent general linear method $(A,U,B,V)$, which is $(L,P)$--symmetric, we  explore a canonical  form of the method based on a diagonal
form of $V$.
The approach is to successively transform $(A,U,B,V)$ to an equivalent method $(A,UT,T^{-1}B,T^{-1}VT)$ and then to
regard this as the base method.  This  leads to a specific form for the coefficient matrices of the method which can
then be back-transformed to a convenient format for practical considerations, such as a requirement that
$U,B,V$ should be real matrices.  As transformations to a canonical form take place, $L$ is also transformed.

Methods in canonical form are convenient to analyze in terms of order of accuracy and the possible presence of parasitic growth factors.

Since $V$ is similar to $V^{-1}$  and each is power-bounded, $T$ exists such that $T^{-1}VT$ is diagonal
with diagonal elements made up from points on the unit circle.  We will see how to carry out this
diagonalization process in such a way that, when the corresponding transformation has also been applied to $U$ and $B$, these matrices have a specific structure.
Because of the original real form of $V$, the diagonal elements are real or come in conjugate pairs.  Hence, in the canonical form,
\begin{equation*}
V = V_0 \oplus V_{-1} \oplus V_1 \oplus V_2 \oplus \cdots \oplus V_n,
\end{equation*}
where
\begin{align*}
V_0 &= \diag(1,1,\dots,1),\\
V_{-1} &= \diag(-1,-1,\dots, -1),\\
V_i &= \diag(\zeta_i, \zeta_i, \dots, \zeta_i, \overline\zeta_i, \overline\zeta_i, \dots, \overline\zeta_i), \quad i=1,2,\dots,n.
\end{align*}
The number of diagonal elements in these blocks  are respectively $m_0$, $m_{-1}$ and $2m_i$.  Because of consistency of the method
$m_0\ge1$, but it is possible that $m_{-1}=0$, indicating that this block is missing.  It is assumed that $m_i\ge1$, although it is possible that $n=0$ indicating that the final $n$ blocks in $V$ do not exist.

To carry out the diagonalization process, define transforming matrices $T$ and $T^{-1}$ of the forms
\begin{equation*}
T=\left[\begin{array}{ccccccc} T_0 & T_{-1} & T_1 &\overline T_1 & \cdots  & T_n &\overline T_n\end{array}\right],\qquad
T^{-1}=\left[\begin{array}{c} S_0 \\ S_{-1} \\ S_1 \\ \overline S_1 \\ \vdots  \\ S_n \\ \overline S_n\end{array}\right],
\end{equation*}
 where, the various submatrices are blocks of eigenvectors; that is
 \begin{align*}
 VT_0 = T_0,\quad VT_{-1} = - T_{-1}, \quad VT_i = \zeta_i T_i,
 \quad V\overline T_i = \overline\zeta_i \overline T_i, \quad i=1,2,\dots, n,\\
 S_0V = S_0,\quad S_{-1}V = - S_{-1}, \quad S_iV_i = \zeta_i S_i,
 \quad \overline S_i V= \overline\zeta_i \overline S_i, \quad i=1,2,\dots, n,
 \end{align*}
 with $T_0, T_{-1}, S_0, S_{-1}$ real.
 Similarly, transformed $B$ and $U$ matrices have the form
 \begin{equation}   \label{eq:BU}
B=\left[\begin{array}{c} B_0 \\ B_{-1} \\ B_1 \\ \overline B_1 \\ \vdots  \\ B_n \\ \overline B_n\end{array}\right],\qquad
U=\left[\begin{array}{ccccccc} U_0 & U_{-1} & U_1 &\overline U_1 & \cdots  & U_n &\overline U_n\end{array}\right].
\end{equation}

In the canonical form, $V^{-1} = \overline V$ and we recall that $L^{-1}=L$, so that \eqref{eq:symm} becomes
\begin{equation*}
L \overline V = VL,
\end{equation*}
and it follows that, in a block representation of $L$ compatible with the block structure of $V$, the off-diagonal blocks
are zero.  Hence we can write
\begin{equation*}
L = L_0 \oplus L_{-1} \oplus L_1 \oplus L_2 \oplus \cdots \oplus L_n.
\end{equation*}
We now consider the structure of the diagonal blocks in $L$.  In the case of $L_0$ and $L_{-1}$,
the idempotent property implies that these matrices are similar to diagonal matrices of the form
$I\oplus (-I)$ where the dimensions of the $+I$ blocks and the $-I$ blocks  are not necessarily the same.  Hence, 
by imposing additional transformations on the method if necessary, we can assume this
diagonal form for $V_0$ and $V_{-1}$.  For the blocks $V_i$, $i=1,2,\dots, n$, the equation
$L_i \overline V_i = V_i L_i$ implies that there exists a non-singular $m_i \times m_i$ matrix
$K_i$ such that
\begin{equation}\label{eq:LiK}
L_i =
\left[
\begin{array}{cc}
0 & K_i\\ K_i^{-1} & 0
\end{array}
\right].
\end{equation}
The choice of the non-singular matrix $K_i$ is arbitrary.  To see why this is the case,
apply the transformation 
\begin{equation}\label{eq:Ktran}
V_i \longmapsto
\left[\begin{array}{cc} I & 0\\ 0 & K_i\end{array}\right]
V_i
\left[\begin{array}{cc} I & 0\\ 0 & K_i^{-1}\end{array}\right],
\end{equation}
which leaves $V_i$ unchanged.  The transformation \eqref{eq:Ktran}
applied to $L_i$ gives
\begin{align*}
L_i \longmapsto & \left[\begin{array}{cc} I & 0\\ 0 & K_i\end{array}\right]
 \left[
\begin{array}{cc}
0 & K_i\\ K_i^{-1} & 0
\end{array}
\right]
\left[\begin{array}{cc} I & 0\\ 0 & K_i^{-1}\end{array}\right]=
  \left[\begin{array}{cc} 0 &I \\  I & 0\end{array}\right],
\end{align*}
so that $K_i$ has been replaced by $I$. We will take the canonical form of $L_i$ to be \eqref{eq:LiK}
with $K_i=I$.

Using the new basis, with $K_i=I$ gives
\begin{equation*}
U\longmapsto UT\qquad\mbox{ and }\qquad   B\longmapsto T^{-1}B,
\end{equation*}
where, as indicated above,  we now use $U$ and $B$ for the transformed matrices. A rearrangement of the symmetry conditions \eqref{eq:symm} now yields
\begin{align}
 B=VLBP,\quad U=PULV.\label{eq:symm4}
\end{align}
Using the canonical forms of $V$ and $L$, and taking $j\geq 1$,  \eqref{eq:symm4} implies that
\begin{equation*}
\left[\begin{array}{c}B_{j}\\
\overline{B}_j\end{array}\right]=\left[\begin{array}{cc} 0 &\zeta_j I \\  \overline{\zeta}_jI & 0\end{array}\right]\left[\begin{array}{c}B_{j}P\\
\overline{B}_jP\end{array}\right], \quad [U_j,\,\overline{U_j}]= [PU_{j},\,P\overline{U}_j]\left[\begin{array}{cc} 0 &\overline{\zeta}_jI \\  \zeta_j I & 0\end{array}\right],
\end{equation*}
for the submatrices $B_j\in \C^{m_i\times s}$,\,$U_j\in \C^{s\times m_j}$.   This simplifies to
\begin{equation}
B_{j}=\zeta_j \overline{B}_{j}P,\qquad U_{j}=\zeta_j P\overline{U}_{j}.\label{eq:b1u1}
\end{equation}
If $P$ represents  the  permutation $\pi$, then the components of $B_j$ and $U_j$ satisfy
\begin{equation}
B_{ji}=\zeta_{j} \overline{B}_{j\pi(i)},\qquad U_{ij}=\zeta_{j} \overline{U}_{\pi(i)j},\qquad   \quad 1\leq i\leq s.
\end{equation}

\subsection{Formulation in real form}
Having constructed a method in canonical form, it is desirable to transform it back to
a formulation in which $B$,. $U$ and $V$ have only real elements.  Consider
complex blocks $B_i, \overline B_i$ and $U_i, \overline U_i$ in \eqref{eq:BU},
corresponding to $V_i =\diag(\zeta I,\overline \zeta I)$.  We will show how it is
possible to construct $T_i$ so that 
\begin{equation*}
T_i^{-1} \left[\begin{array}{c} B_i\\ \overline B_i \end{array}\right], \qquad
\left[\begin{array}{cc} U_i & \overline U_i \end{array}\right] T, \qquad
T_i^{-1} V T
\end{equation*}
are each real.
The suggested choice of $T_i$ and $T_i^{-1}$ are
\begin{equation*}
T_i = \left[\begin{array}{rr} I &  \ii I\\ I &- \ii I\end{array}\right],\qquad
T_i^{-1} = \tfrac12\left[\begin{array}{rr} I & I\\
-\ii I  & \ii I\end{array}\right],
\end{equation*}
leading to transformed blocks
\begin{align*}
T_i^{-1} \left[\begin{array}{c} B_i\\ \overline B_i \end{array}\right]&= 
\left[\begin{array}{c} \re B_i\\ \im B_i \end{array}\right],\\
\left[\begin{array}{cc} U_i & \overline U_i \end{array}\right] T &=
\left[\begin{array}{rr} 2\re U_i & -2\im U_i \end{array}\right],\\
T_i^{-1} \diag(\zeta_iI, \overline \zeta_iI) T_i &=
\left[\begin{array}{rr} \re\zeta_i I& -\im\zeta_i I\\ -\im\zeta_i I & \re\zeta_i I\end{array}\right].
\end{align*}


\section{Stability}\label{sec:stab}
 
\subsection{Linear stability}

\vspace*{.05in} 

\begin{defn}\label{def:linstab}
For a method $(A,\,U,\,B,\,V)$ and $Z=\diag(z_1,\,\ldots,\,z_s)\in \C^{s\times s}$ such that $I-AZ$ is non--singular, the linear stability function is given by
\begin{equation}
M(Z):=V+BZ(I-AZ)^{-1}U.\qquad\label{eq:linstab}
\end{equation}
\end{defn}

\begin{theorem}\label{theo:MZ}
Method $(A,\,U,\,B,\,V)$ is symmetric, if and  only if there exists a permutation matrix $P$ with  $P^2=I$ such that for all diagonal $Z\in \C^{s\times s}$ with $\|A\|\|Z\|<1$,
\begin{equation}
LM(-PZP)LM(Z)=I.\label{eq:lstab}
\end{equation}
\end{theorem}
\begin{proof} ({\bf only if}) Assume first that $A$ is $S$--irreducible, see \cite{hs81}.   Choose  $y^{[0]}\in \C^s$ and diagonal $Z\in \C^{s\times s}$ such that $\|A\|\|Z\|< 1$. Let $Y$ be  the unique solution of 
\begin{equation*}
(I-AZ)Y=Uy^{[0]}.
\end{equation*}
For almost all $Z$, $S$ irreducibility implies $Y_i\neq Y_j$ implies $i\neq j$, $1\leq i,\,j\leq s$. Using interpolation, we may construct continuous $f:\C\longrightarrow \C$ such that 
\begin{equation*}
f(Y_i)=Z_{ii}Y_i,\qquad 1\leq i\leq s.
\end{equation*}
Let $h=1$ and set $y^{[1]}:=BF+Vy^{[0]}$. Then,
\begin{equation*}
y^{[1]}=\M_h y^{[0]}=M(Z)y^{[0]}. 
\end{equation*}
Also, by  \eqref{eq:symm}, $y^{[0]}=\M_h^{-1}y^{[1]}=LM(-PZP)Ly^{[1]}.$ Thus,
\begin{equation*}
 y^{[0]}=LM(-PZP)Ly^{[1]}=LM(-PZP)LM(Z)y^{[0]}.  
\end{equation*}
Hence, \eqref{eq:lstab} holds for almost all diagonal $Z$, $\|Z\|\|A\|<1$, when $A$ is $S$--irreducible. 
The general case follows from the continuity of $LM(-PZP)LM(Z)$ and the density of  $Z$ and $S$--irreducible $A$.  

\noindent
{\bf (if)} {\em Sketch} Let  $Z=\epsilon$ diag$(x)$ for non-zero $x\in \C^s$, and expand \eqref{eq:lstab} in powers of $\epsilon$. The identities in the $4$ quadrants of \eqref{eq:symm} follow from the terms of $O(1)$, $O(\epsilon)$ and $O(\epsilon^2)$.
\end{proof}

\noindent
{\bf Remark:} This result, which generalizes a linear multistep theorem of \cite{es92}, is in the spirit of the $AN$--stability characterization of algebraic stability \cite{b87}. In the Runge--Kutta case, $L=1$, and \eqref{eq:lstab} generalizes the known necessary condition for symmetry: $M(-zI)M(zI)=I$ \cite[V.6]{hlw}. 
 
\begin{lemma}\label{lem:eiginv}
Suppose that the method $(A,\,U,\,B,\,V)$  is $(L,\,P)$--symmetric and has real coefficients. Suppose also that $R\in  \R^{s\times s}$ is a diagonal matrix such that $ R=P RP$. Then,
\begin{equation*}
\zeta\in \sigma(M(\ii R)) \quad \Longrightarrow \quad 
\overline{\zeta}^{-1}\in \sigma(M(\ii R)).
\end{equation*}
\end{lemma}
\begin{proof} 
The symmetry of $M$ and the assumption on $ R$ imply that 
\begin{equation*}
LM(\ii R)LM(-\ii P RP)=I.
\end{equation*}
Hence, $M(\ii R)$ is of full rank and possesses an inverse. In particular, $\zeta\neq 0$ and $\zeta^{-1}\in \sigma(M(\ii R)^{-1})$. Since
\begin{equation*}
LM(\ii R)^{-1}L=M(-\ii P RP)=M(-\ii R),
\end{equation*}
similarity implies that $\zeta^{-1}\in \sigma(M(-\ii R))$. Taking the complex conjugate, it follows that $\overline{\zeta}^{-1}\in \sigma(M(\ii R))$.
\end{proof}

\begin{theorem}\label{theo:imag}
Assume that the method $(A,\,U,\,B,\,V)$ is  symmetric and has real coefficients. Assume also that the eigenvalues of $V$ are distinct.  Then,  there exists $k_0>0$ such that the  eigenvalues of $M(\ii R)$ are distinct and unimodular for all diagonal $ R\in  \R^{s\times s}$ such that $ R=P RP$ and $\| R\|<k_0$. In particular, the linear stability domain $S$ contains an imaginary interval $(-\ii k_0,\,\ii k_0)$.
\end{theorem}
\begin{proof} 
The eigenvalues of $V$ are unimodular, so $\zeta\in \sigma(V)$ implies $\zeta-\overline{\zeta}^{-1}=0$. Let $\delta>0$ be  the closest distance between any two eigenvalues of $V$. For small diagonal $ R\in \R^{s}$, the eigenvalues of $M(\ii R)$ are  continuous functions of $ R$. Thus,  there exists $k_0>0$ such that $\| R\|< k_0$ implies\\
(i) no eigenvalue of $M(\ii R)$ is closer than $\delta/2$ to any other eigenvalue;\\
(ii) $\zeta\in \sigma(M(\ii R))$ implies $|\zeta-\overline{\zeta}^{-1}|<\delta/4$.

Now, suppose that diagonal $ R\in  \R^{s\times s}$ satisfies $\| R\|<k_0$ and that $\zeta\in \sigma(M(\ii R))$. Then, Lemma \ref{lem:eiginv} implies that $\overline{\zeta}^{-1}\in \sigma(M(\ii R))$. Furthermore, conditions (i) and (ii) imply that $\overline{\zeta}^{-1}=\zeta$; i.e. $|\zeta|=1$.

If $R=xI$, $|x|\leq \|xI\|<k_0$, then the foregoing results imply that all eigenvalues of $M(\ii xI)$ are unimodular. Thus, $M(\ii xI)$ is power--bounded and $\ii x\in S$.
\end{proof}

\noindent
{\bf Remark:} The continuity argument used in the proof of Theorem \ref{theo:imag} may be used to increase $k_0$ until $M(\ii k_0)$  is ill--defined or has a multiple eigenvalue. 




\subsection{Parasitism}

For a zero--stable symmetric method, $V$ is power--bounded and similar to $V^{-1}$. Hence, all of the eigenvalues of $V$ are unimodular and, at worst, semi--simple. Typically, however, symmetric methods will be applied to problems without overall growth or decay. Hence, care   is needed to limit the growth of  components of the numerical solution associated with the non--principal  eigenvalues of $V$. Below, it is assumed that the method is written in the canonical coordinates of Subsection \ref{subs:V}.

\begin{defn}
 A preconsistent symmetric method is said to be {\bf parasitism--free} if there exist $C,\,\nu>0$ such that, given $\epsilon>0$,   
\begin{equation}
\|((I- e_1 e_1\tr)M(zI)(I- e_1 e_1\tr))^n\|\leq C(1+\nu \epsilon^{2})^{n},\qquad \quad n\in \N,\label{eq:parfree}
\end{equation}
for all $z\in \C$ such that $|z|<\epsilon$.
\end{defn}

\begin{prop}\label{prop:parfree}
A preconsistent symmetric method is parasitism-free if and only if
\begin{equation}
 w\H_{\zeta}BUu_{\zeta}=0,\label{eq:BUzero}
\end{equation}
whenever $\zeta$ is a non--principal eigenvalue of $V$, and $u_\zeta$ and $w_\zeta\H$ are respectively right and left eigenvectors corresponding to $\zeta$.
\end{prop}

\noindent
{\bf Remark:} In the assumed canonical coordinates,  \eqref{eq:BUzero} implies the simple condition,
\begin{equation}
(BU)_{ii}=0,\qquad 2\leq i\leq r,\label{eq:BUzero2}
\end{equation}
If $\zeta$ is a multiple eigenvalue of $V$, \eqref{eq:BUzero} implies that some  off--diagonal elements are also zero.

\begin{proof} If \eqref{eq:BUzero} holds, then there exist $r$ eigentriples $(\zeta,\,u_{\zeta},\,w\H_{\zeta})$ for $V$ such that
\begin{equation*}
Vu_{\zeta}=\zeta u_{\zeta},\quad w\H_{\zeta}V=\zeta w\H_{\zeta},\quad  w\H_{\zeta}u_{\zeta}=1.
\end{equation*}
Set $ T=[u_{\zeta_1},\,\ldots,\,u_{\zeta_r}]$. Then, $ T^{-1}=[w_{\zeta_1},\,\ldots,\,w_{\zeta_r}]$ and
$ T^{-1}V T=\mbox{diag}(\zeta_1,\,\ldots,\,\zeta_{r}).$ Assuming $\zeta_1$ is the principal eigenvalue,
\begin{align}
& T^{-1}(I- e_1 e_1\tr)M(zI)(I- e_1 e_1\tr) T= T^{-1}(I- e_1 e_1\tr)(V+zBU+O(|z|^2))(I- e_1 e_1\tr) T\nonumber\\
&\qquad =\mbox{diag}(0,\,\zeta_2,\,\ldots,\,\zeta_r)+z T^{-1}(I- e_1 e_1\tr)BU(I- e_1 e_1\tr) T+O(|z|^2)\label{eq:pardiag}.
\end{align}
For small $z\in \C$, eigenvalue perturbation theory (Wilkinson 1965) implies that the eigenvalues of 
$(I- e_1 e_1\tr)M(zI)(I- e_1 e_1\tr)$ consist of a term of $O(|z|^2)$, corresponding to the principal eigenvector of $V$, and $\{\zeta_{j}(z)\}_{j=2}^{r}$, where
\begin{equation*}
\zeta_{j}(z)=\zeta_{j}(0)+z w_{\zeta_j}\H BUu_{\zeta_j}+O(|z|^2)=\zeta_{j}(0)+O(|z|^2),\qquad 2\leq j\leq r.
\end{equation*}
Since $|\zeta_j(0)|=1$, the parasitism--free condition \eqref{eq:parfree} is satisfied.

Conversely, assume that \eqref{eq:BUzero} is not satisfied, and let small $\epsilon>0$ be chosen. Set $z=\delta \zeta/w\H_{\zeta}BUu_{\zeta}$, where  $\delta>0$ is such that $|z|=\epsilon$.  Then, similarly to  \eqref{eq:pardiag},
\begin{equation*}
w_{\zeta}\H((I- e_1 e_1\tr)M(zI)(I- e_1 e_1\tr))^n u_{\zeta}=\zeta^n(1+n\epsilon+O(\epsilon^2)),
\end{equation*}
and so \eqref{eq:parfree} cannot be satisfied.
\end{proof}

\begin{corollary}\label{theo:noexplicit}
There are no explicit consistent symmetric parasitism-free  methods.
\end{corollary}
\begin{proof} 
Consistency \eqref{defn:cons}, with $u=w=e_1$, implies that
\begin{equation*}
(BU)_{11}= e_1\tr BU e_1= e_1\tr B\1= e_1\tr(I-V)v+ e_1\tr  e_1=1.
\end{equation*}
Combined with \eqref{eq:BUzero2}, the canonical form for $V=$ diag$(1,\,\zeta_{2},\,\dots,\,\zeta_r)$, and the first quadrant of \eqref{eq:symm}, we deduce that
\begin{equation}
1=\sum_{i=1}^{r}(BU)_{ii}  \zeta_{i}^{-1}=\mbox{tr}(BUV^{-1})=\mbox{tr}(UV^{-1}B)=\mbox{tr}(A+PAP). \label{eq:trace}
\end{equation}
Hence,  diag$(A)\neq 0$. Thus, the method has at least one implicit stage.
\end{proof}

\noindent
{\bf Remark:} As mentioned in the Introduction, it is known that all symmetric Runge--Kutta methods are implicit, and that all symmetric linear multistep methods suffer from parasitism, (whether or not they are explicit). An example in Section \ref{sec:examp} show that only one implicit stage is necessary for a general linear method to be symmetric and parasitism--free.

\subsection{Transfer function characterization of symmetric methods}

\vspace*{.05in} 

\begin{defn}\label{defn:transf} For a method $(A,\,U,\,B,\,V)$, and $\zeta\in \C$ such that $\zeta I-V$ is nonsingular,  the transfer function is defined by
\begin{equation}
N(\zeta):=A+U(\zeta I-V)^{-1}B.\label{eq:tranf}
\end{equation}
\end{defn}
This function has previously been considered in \cite{b87} and \cite{hil06} in the context of algebraically stable methods. We omit the proof of the following straightforward result.

\begin{lemma}[\cite{b87}]\label{lem:Nuniq}
Given GLMs $(A,\,U,\,B,\,V)$ and $(\widehat{A},\,\widehat{U},\,\widehat{B},\,\widehat{V})$, with diagonalizable $V$ and $\widehat{V}$,
\begin{equation*}
\widehat{N}(\zeta)=N(\zeta),\qquad\zeta\in \C\setminus (\sigma(V)\cup\sigma(\widehat{V})\cup\{0\}),
\end{equation*}
if and only if there exists non--singular $T\in \C^{r\times r}$ such that
\begin{equation}
\left[\begin{array}{cc}\widehat{A}&\widehat{U}\\\widehat{B}&\widehat{V}\end{array}\right]=\left[\begin{array}{cc}A&UT\\T^{-1}B&T^{-1}VT\end{array}\right].\label{eq:Nuniq}
\end{equation}
\end{lemma}

\begin{theorem}\label{theo:nyqsymm}
A method $(A,\,U,\,B,\,V)$ is symmetric if and only if  there exists a permutation matrix $P$ such that $P^2=I$ and
\begin{equation}
N(\zeta)=-PN(\zeta^{-1})P,\qquad \zeta\in \C\setminus \Delta_0,\label{eq:nyqsymm}
\end{equation}
where  $\Delta_0:=\{\zeta\in \C\,:\, |\zeta|=1 \;\mbox{ or }\;\zeta=0\}.$
\end{theorem}

\vspace*{.05in}

\noindent
{\bf Remark:} Identity \eqref{eq:nyqsymm} is an $L$--free characterization of symmetry, which generalizes the $(\sigma/\rho)(\zeta)=-(\sigma/\rho)(\zeta^{-1})$ condition \cite{es92}  for multistep symmetry.
\begin{proof} ({\bf only if}) Given a method $(A,\,U,\,B,\,V)$, the method $(A^{*},\, U^{*},\, B^{*},\, V^{*})$ appearing on the right--hand side of \eqref{eq:symm} is the  {\em adjoint method}, see \cite{hlw}. From formula \eqref{eq:tranf}, 
 \begin{align}
 N^{*}(\zeta)&=P(UV^{-1}B-A)P+PUV^{-1}L(\zeta I-LV^{-1}L)^{-1}LV^{-1}BP\nonumber\\
&=P(-A+U(I+(\zeta V-I)^{-1})V^{-1}B)P\nonumber\\
&=-P(A+U(\zeta^{-1}I-V)^{-1}B)P=-PN(\zeta^{-1})P,\label{eq:nyqadj}
\end{align}
for $ \zeta\in \C\setminus \Delta_0$. For a symmetric method, \eqref{eq:symm} implies that $( A^{*},\,U^{*},\,B^{*},\,V^{*})=(A,\,U,\,B,\,V)$. Hence, \eqref{eq:nyqadj} implies \eqref{eq:nyqsymm}.

\noindent
({\bf if}) Now assume that  \eqref{eq:nyqsymm} holds for method $(A,\,U,\,B,\,V)$, and let $( A^{*},\,U^{*},\,B^{*},\,V^{*})$ denote its adjoint for $L=I$. From identity \eqref{eq:nyqadj}, we know that
\begin{equation*}
N(\zeta)=-PN(\zeta^{-1})P= N^{*}(\zeta),\qquad \zeta\in \C\setminus \Delta_0.
\end{equation*}
Applying Lemma \ref{lem:Nuniq},  there exists non--singular $T\in \C^{r\times r}$ such that
\begin{equation}
\left[\begin{array}{cc} A& U\\
B&V\end{array}\right]=\left[\begin{array}{cc}P(UV^{-1}B-A)P& PUV^{-1}T\\T^{-1}V^{-1}BP&T^{-1}V^{-1}T\end{array}\right].\label{eq:Tequiv}
\end{equation}
Using a diagonal decomposition of $V$, as in Subsection \ref{subs:V},  $T$ may be altered if necessary so that $T^2=I$ on each eigensubspace of $V$, without affecting identity \eqref{eq:Tequiv}. Thus, \eqref{eq:Tequiv} holds for $T=L$ auch that $L^2=I$.
\end{proof}

\subsection{A transfer function characterization of $G$-symplectic methods}\label{subs:Gdefn}
It is the purpose of symplectic, or canonical, one-step methods to preserve the value of $[y_n, y_n]_Q$ as $n$ increases, where the symmetric bi-linear function $[\cdot, \cdot]_Q$ is defined by
\begin{equation*}
[y,z]_Q:= \langle y, Qz\rangle,
\end{equation*}
$Q$ is a symmetric $N\times N$ matrix, and $\langle \cdot,\,\cdot\rangle$ is an inner product on $X=\R^N$.  If $[y,f(y)]_Q=0$, then $[y(x),y(x)]_Q$ is an invariant of the ODE \eqref{eq:ODE}.

For a general linear method \eqref{eq:method}, it is necessary to work in the higher dimensional space $X^r$ and we  consider the possible preservation
of $[y^{[n]},y^{[n]}]_{G\otimes Q}$ as  $n$ increases, where 
\begin{equation*}
[y, z]_{G\otimes Q} = \sum_{i,j=1}^r g_{ij} [y_i,z_j]_Q,\qquad y,\,z\in X^r,\qquad y_1,\ldots,\,y_r\in X,
\end{equation*}
and it will always be assumed that $G\in \C^{r\times r}$ is Hermitian and non-singular.
It is known \cite{hlw} that the conditions for $[y^{[n]},y^{[n]}]_{G\otimes Q}=[y^{[n-1]},y^{[n-1]}]_{G\otimes Q}$ are that 
there exists a real  diagonal $s\times s$ matrix  $D$ such that
\begin{equation}
M:=\left[
\begin{array}{cc} DA+A\H D - B\H GB&DU- B\H GV \\ U\H D - V\H GB &G-V\H G V   \end{array}
\right]=0.\label{eq:Mmat}
\end{equation}
Note that $A$ is assumed to remain real, but the other coefficient matrices may become complex--valued under a complex coordinate transformation $T$. Below, the method is assumed to be expressed in the canonical coordinates of Subsection \ref{subs:V}.

\begin{theorem}\label{theogree}
 Let  $(A,\,U,\,B,\,V)$ be a consistent method  with real non--singular diagonal matrix $D=${\rm diag}$(B\H e_1)$.  Then, the method is $G$--symplectic if and only if
 \begin{equation}
 N(\zeta)=-D^{-1}N\H(\zeta^{-1})D,\qquad \zeta\in \C\setminus \Delta_0.\label{eq:nyqgsymp}
\end{equation}
\end{theorem}
\noindent
(Here, $N\H(\zeta^{-1})$ means evaluate the matrix function $N\H$ for the argument $\zeta^{-1}$.)

\noindent
{\bf Remark:} Identity \eqref{eq:nyqgsymp} is a $G$--free characterization of $G$--symplecticity. In the linear multistep case \cite{es92}, this is the same as \eqref{eq:nyqsymm}.
\begin{proof} ({\bf only if}) For $\zeta \in \C\setminus(\mbox{eig}(V)\cup\Delta_0)$, \eqref{eq:Mmat} implies that
\begin{eqnarray}
0&=&[I, \;\;B\H(\zeta^{-1}-V\H)^{-1}]\left[\begin{array}{cc}DA+A\H D-B\H GB& DU-B\H GV\\
U\H D-V\H GB& G-V\H GV\end{array}\right]\left[\begin{array}{c}I\\(\zeta I-V)^{-1}B\end{array}\right]\nonumber\\
&=&[DA+DU(\zeta I-V)^{-1}B]+[A\H D+B\H(\zeta^{-1}I-V\H)^{-1}U\H D]\nonumber\\
& &\quad +(1-\zeta.\overline{\zeta}^{-1})B\H(\zeta^{-1}-V\H)^{-1}G(\zeta I-V)^{-1}B\nonumber\\
&=& DN(\zeta)+N\H(\zeta^{-1})D.
\end{eqnarray}
({\bf if})  From \eqref{eq:nyqadj} and Lemma \ref{lem:Nuniq} it follows that there is a nonsingular $T\in \C^{r\times r}$ such that
\begin{equation}
\left[\begin{array}{cc}A&U\\B&V\end{array}\right]=\left[\begin{array}{ccc}D^{-1}(B\H V^{-{\mbox{\tiny\sf\bfseries H}}} U\H-A\H )D & \;\;& D^{-1}B\H V^{-{\mbox{\tiny\sf\bfseries H}}}T\\
T^{-1} V^{-{\mbox{\tiny\sf\bfseries H}}}U\H D& & T^{-1}V^{-{\mbox{\tiny\sf\bfseries H}}}T\end{array}\right].\label{eq:preGsymp}
\end{equation}
From the $(2,1)$ quadrant, $B=T^{-1}V^{-{\mbox{\tiny\sf\bfseries H}}}U\H D$, and hence $DUV^{-1}=B\H T^{H}$. 
From the $(1,2)$  and $(2,2)$  quadrants, $DUV^{-1}=B\H T$. Thus,
\begin{equation*}
0=DUV^{-1}-DUV^{-1}=B\H(T-T\H).
\end{equation*}
From the $(2,2)$  quadrant, $TV=V^{-{\mbox{\tiny\sf\bfseries H}}}T$, which implies $T^{H}V=V^{-{\mbox{\tiny\sf\bfseries H}}}T\H$. Hence, 
\begin{equation*}
\tfrac12(T+T\H)V=V^{-{\mbox{\tiny\sf\bfseries H}}}\tfrac12(T+T\H).
\end{equation*}
Hence, $\frac12(T+T\H)$ may be substituted for $T$ in \eqref{eq:preGsymp}; i.e. $T$ may be assumed to be Hermitian. We now observe that  \eqref{eq:preGsymp}  implies \eqref{eq:Mmat}, with $G=T$.
\end{proof}

\subsection{Methods that are both symmetric and $G$--symplectic}
\rule{0in}{.01in}\\
\begin{theorem}\label{theo:symmsymp}
If a  GLM   satisfies two of the following conditions, it satisfies all three:\\
(i) The method is symmetric;\\
(ii) The method is $G$--symplectic;\\
(iii) There exists a non--singular  $T\in \C^{r\times r}$ such that
\begin{equation}
\left[\begin{array}{cc} PDA&PDU\\ TB& TV\end{array}\right]=\left[\begin{array}{cc}(PDA)\H&B\H T\\ (PDU)\H& V\H T\end{array}\right].\label{eq:PDN2}
\end{equation}
Condition (iii) is equivalent to
\begin{equation}
PDN(\zeta)=(PDN(\zeta))\H,\qquad \zeta\in \C\setminus \Delta_0.\label{eq:PDNsymm}
\end{equation}
\end{theorem}
 
 \begin{proof} The equivalence of \eqref{eq:PDN2} and \eqref{eq:PDNsymm} follows from Lemma \ref{lem:Nuniq}. On the other hand symmetry and $G$--symplecticity are respectively equivalent to the transfer function indentities  \eqref{eq:nyqsymm} and  \eqref{eq:nyqgsymp}. The following equivalences for the transfer function  complete the proof:
\begin{equation*}
\mbox{\eqref{eq:PDNsymm}} + \mbox{\eqref{eq:nyqgsymp}}\Longleftrightarrow \mbox{\eqref{eq:PDNsymm}} + \mbox{\eqref{eq:nyqsymm}}\Longleftrightarrow \mbox{\eqref{eq:nyqsymm}}+\mbox{\eqref{eq:nyqgsymp}}.
\end{equation*}
\end{proof}

The following closely connected result, the proof of which we omit, is useful in the construction of methods that are both symmetric and $G$--symplectic. The canonical coordinates of Subsection \ref{subs:V} are assumed.
\begin{theorem}\label{theo:gsym}
Consider a consistent $(L,\,P)$--symmetric general linear method,  where  $A$ is lower triangular and $P$ is the reversing permutation matrix; i.e. $(Pv)_i=v_{s+1-i}$, $1\leq i\leq s$, for $v\in \C^s$.  Then, the method is G-symplectic if\\
{\rm(i)} non--zero real scalars $h_1,\, h_2, \dots, h_r$ exist such that 
\begin{equation}\label{eq:hi}
DUe_i = h_i \zeta_i B\H e_i,\qquad 1\leq i\leq r,
\end{equation}
where $h_1$ is such that $D=h_1$ {\rm diag}$(B\H e_1)$.\\
{\rm(ii)} The diagonal  part of $A$ satisfies
\begin{equation}
\diag(a_{11}, a_{22}, \dots, a_{s-1,s-1}, a_{ss})=\diag(a_{ss}, a_{s-1,s-1}, \dots, a_{22}, a_{11} ).\label{eq:adiag}
\end{equation}
If the eigenvalues of $V$ are distinct and {\rm diag}$(D)$ has no zero elements, then conditions \eqref{eq:hi} and \eqref{eq:adiag} are also necessary for $G$--symplecticity.
\end{theorem}


\def\FF{\mbox{\sf\bfseries F}}

\section{Symmetry and even order results}\label{sec:even}
\subsection{Even order for the general linear method}
The method $\M_h$ is of order $p\in {\mathbb N}$ relative to the starting method $S_h$ if

\begin{equation}
S_{h}E_{h}y_0-\M_{h}S_h y_0=C_{p+1}(y_0)h^{p+1}+ O(h^{p+2}),
\label{eq:order}
\end{equation}
where, for  $T_{p+1}$ the set of rooted trees of order $p+1$, elementary differentials $\FF(t)(y_0)\in X$, symmetry coefficients $\sigma(t)\in \R$ and weight vectors $\Psi(t)\in \C^r$,
\begin{equation*}
C_{p+1}(y_0):=\sum_{t\in T_{p+1}}\Psi(t)\frac{\FF (t)(y_0)}{\sigma(t)}\in X^r.
\end{equation*}
The order of the method $\M_h$ is  $p\in {\mathbb N}$, if $p$ is the greatest integer  such that there is an $S_h$ relative to which $\M_h$ has order $p$.

Following the work in Subsection \ref{subs:V}, we assume that the method may be written in coordinates such that $V$,  $B$ and $U$ take  the form
\begin{equation}
V = \left[\begin{array}{cc}  1 & \0\tr\\ \0 & \dot V  \end{array}\right], \qquad B = \left[\begin{array}{c}  b\tr \\ \dot B \end{array}\right],\qquad 
U = \left[\begin{array}{cc}  \1 & \dot U  \end{array}\right]. \label{eq:bldiag}
\end{equation}
In particular, we note that $I_{r-1}-\dot V$ is non-singular.

We assume that the method is of 
of order $p$ relative to the starting method $S_h$. Written in the new basis, the principal component of $S_h$ is represented by the
B-series $\zeta$, (see \cite{hnw}). The remaining components are given by the vector of  B-series, $\xi$.   For some $\eta$, representing the stage values, the stage equations and the update equations for the principal and non--principal components may be written in terms of B-series: 
\begin{align}
\eta(t) &= A(\eta D)(t) +\1 \zeta(t) + \dot U\xi(t),\label{eq:eta2}\\
(E\zeta)(t) &=  b\tr (\eta D)(t) + \zeta(t),\label{eq:Ez2}\\
(E\xi)(t) &= \dot B(\eta D)(t) + \dot V\xi(t),\label{eq:Ex2}
\end{align}
for all $t$ such that $|t|\le p$. Suppose that a second starting method $\widehat{S}_{h}$ is similarly represented by B-series  $\widehat \zeta$ and $\widehat \xi$.

\begin{lemma}\label{lem:Suniq}
Suppose that the method $\M_h$ is of order $p$ relative to $\S_h$ and also of order
$p$ relative to $\widehat{S}_{h}$,  and that $\zeta(\ta)=\widehat{\zeta}~(\ta)$. Then,
\begin{align*}
\widehat\zeta(t) &=\zeta(t),\quad |t| \le p-1,\\
\widehat\eta(t) &=\eta(t),\quad |t| \le p-1,\\
\widehat\xi(t) &=\xi(t),\quad |t| \le p.
\end{align*}
where $\widehat \eta$ is defined by \eqref{eq:eta2} but for the
starting method $[\widehat \zeta, \widehat \xi]$.
\end{lemma}
\begin{proof}
We first recall and extend some notation on trees.  If
$|t_1|, \dots, |t_n| \geq 1$ then 
\begin{equation}
t= [\tau^m t_1 t_2 \cdots t_n] \label{eq:taum}
\end{equation}
denotes a rooted tree with order
\begin{equation*}
|t| = 1+m + |t_1|+|t_2|+\cdots+ |t_n|
\end{equation*}
formed by joining the roots of $m$ copies of $\tau$ and
each of the roots of $t_i$ ($i=1,2,\dots,n$) to a new root.

The valency of the root of $t$,  will be written as
\begin{equation*}
w(t) = m+n.
\end{equation*}
The binary product of trees will be used in the special case
\begin{equation*}
t\tau =  [\tau^{m+1} t_1 t_2 \cdots t_n],
\end{equation*}
where $t$ is given by \eqref{eq:taum}.   Note that $w(t\tau) = w(t)+1$.

If $\eta$ is the B-series representing stage values of a general 
linear method, then for this same $t$, the B-series for the stage derivatives 
are given by
\begin{equation*}
(\eta D)(t) = \eta(\tau)^m\prod_{i=1}^n \eta(t_i),
\end{equation*}
where the powers and products on the right-hand side are componentwise.
We will prove by induction on $k=1,2,\dots,p-1$, 
\begin{align}
\widehat\zeta(t) &=\zeta(t),\quad |t| \le k,\label{eq:hatz}\\
\widehat\eta(t) &=\eta(t),\quad |t| \le k,\label{eq:hate}\\
\widehat\xi(t) &=\xi(t),\quad |t| \le k+1.\label{eq:hatx}
\end{align}
Note that \eqref{eq:hatz} and \eqref{eq:hate} are true when $k=0$, and (i) 
follows from \eqref{eq:Ex2} by substituting the tree $t=\tau$
to give
\begin{equation*}
 \xi(\tau) =(I-\dot V)^{-1} \dot B \1,
\end{equation*}
with the same result for $\widehat \xi(\tau)$.
 Now assume the result for integers less than $k$,
 and we prove  \eqref{eq:hatz}  for a specific $k\in\{1,2,\dots,p-1\}$.  For $k=1$,
 this holds by assumption.  For $k>1$, consider each tree
 $t$  of order $k$ in a sequence in which $w(t)$ is non-increasing.
 For $t$ given by \eqref{eq:taum}, substitute $t\tau$ into \eqref{eq:Ez2} to give the result
 \begin{equation*}
(m+1) \zeta(t) =  b\tr \eta(\tau)^{m+1}\prod_{i=1}^n \eta(t_i)-C(t,\zeta),
 \end{equation*}
 where $C(t,\zeta)$ involves trees already considered for lower $k$ and for
 trees with this same order which occurred earlier in the sequence.
 Obtain a similar result for $\widehat\zeta$ and note that the terms
 on the right-hand side are identical in the two cases.  The result \eqref{eq:hate}
 follows from \eqref{eq:eta2} and the corresponding formula for $\widehat\eta(t)$.
 To prove \eqref{eq:hate} for any tree of order $k+1$, use \eqref{eq:Ex2} to obtain a formula
 for $(I-\dot V) \xi(t)$ with the same result for $(I-\dot V) \widehat \xi(t)$.
\end{proof}

\noindent
{\bf Remarks:} (i) The proof of Lemma \ref{lem:Suniq} serves as a constructive proof of the existence of $\S_h$. Note that the order $h^p$ coefficient of $\zeta$ is arbitrary.\\
(ii)  The assumption $\zeta(\ta)=\widehat{\zeta}(\ta)$ can always be assumed because, if it were not true then $\S_h$ can be replaced by $\S_h E_{\theta h}$ for a suitable $\theta\in \R$.
  
\begin{lemma}\label{lem:symmS}
Suppose that the method $\M_h$ is symmetric and of  order $p$ relative to $\S_h$, such that $\zeta(\ta)=0$. Then, $\M_h$ is also of order $p$ relative to both $L\S_{-h}$ and the symmetric starting method $\frac12(S_h + LS_{-h})$. The B-series for all $3$ starting methods agree up to order $p$, except possibly in the first component of the trees of order $p$.
\end{lemma}
\begin{proof}
Consider  \eqref{eq:order} with  $h$ and  $y_0$ replaced by $-h$ and  $y_1=E_h y_0$. A left--{mul\-tip\-lic\-ation} by $L\M_{-h}^{-1}$ then yields 
\begin{equation*}
(L\M_{-h}^{-1}L)(LS_{-h})E_{-h}y_1=LS_{-h}y_1 +  LV^{-1}C_{p+1}(y_1)(-h)^{p+1} +O(h^{p+2}),
\end{equation*}
where we  note that the Fr\'{e}chet derivative of $L\M_{-h}^{-1}$ is  $LV^{-1}+O(h)$.
Symmetry implies $\M_h=L\M_{-h}^{-1}L$; also, $C_{p+1}(y_1)=C_{p+1}(y_0)+O(h)$. Thus,
\begin{equation}
\M_h (LS_{-h})y_0=(LS_{-h})E_h y_0+  LV^{-1}C_{p+1}(y_0)(-h)^{p+1}+O(h^{p+2}),\label{eq:LS}
\end{equation}
and so $\M_h$ is of order $p$ relative to $L\S_{-h}$. Now, by Lemma \ref{lem:Suniq} the B-series for $L\S_{-h}$, and therefore also that for  $\frac12(S_h + LS_{-h})$, agrees with the B-series for $\S_h$ up to order $p$, except possibly in the first component of the trees of order $p$. The proof of Lemma \ref{lem:Suniq} shows that this is sufficient for $\frac12(S_h + LS_{-h})$ to be a starting method relative to which $\M_h$ is of order $p$.
\end{proof}

\begin{lemma}\label{lem:uerror}
If $(\M_{h},\,S_h)$ satisfy \eqref{eq:order} for some $p\in \N$,  then $S_h$ may be chosen so that 
\begin{equation}
S_h E_h y_0=\M_h S_{h}y_0+ K_{p+1}(y_0)e_1h^{p+1}+O(h^{p+2}),\label{eq:uerror}
\end{equation}
where $K_{p+1}(y_0):=e_1\tr  C_{p+1}(y_0)\in X$.
\end{lemma}
\begin{proof} Replace $\S_h$ satisfying \eqref{eq:order}, by $\S_h+\delta_h$ such that
\begin{equation*}
(I-V)\delta_h=-(I-e_1 e_1\tr)C_{p+1}(y_0)h^{p+1},\qquad e_1\tr \delta_h=0.
\end{equation*}
\end{proof}

\begin{theorem}\label{theo:evenloc}
Suppose that  $\M_h$  is a  symmetric consistent method and that $1$ is a simple eigenvalue of $V$. Then, $\M_h$ is of even order $p$, and there is a symmetric starting method $\S_h$ relative to which $\M_h$ is of order $p$. 
\end{theorem}
\begin{proof} Since $\M_h$ is consistent, it is of order $p$, for some $p\in \N$. Lemma \ref{lem:symmS} ensures the existence of a symmetric starting method $\S_h$ relative to which $\M_h$ is of order $p$.  Since $L\S_{-h}=\S_h$, identities \eqref{eq:order}  and \eqref{eq:LS} are the same for this $\S_h$. Equating the terms of order $h^{p+1}$, we obtain 
\begin{equation*}
C_{p+1}(y_0)=(-1)^{p}LV^{-1}C_{p+1}(y_0).
\end{equation*}
By Lemma \ref{lem:uerror}, $C_{p+1}(y_0)=K_{p+1}(y_0)e_1$. As $\M_h$ is of order $p$, the term $K_{p+1}(y_0)$ is non--zero. From Subsection \ref{subs:V}, $LV^{-1}e_1=e_1$. Thus,
\begin{equation*}
K_{p+1}(y_0)=e_1\tr C_{p+1}(y_0)=(-1)^{p}e_1\tr LV^{-1}C_{p+1}(y_0)=(-1)^{p}K_{p+1}(y_0).
\end{equation*}
Hence, $p$ is even.
\end{proof}

\begin{lemma}\label{lem:fin}
Suppose that $1$ is a simple eigenvalue of $V$ and that $\M_h$ is of order $p$ relative to $\S_h$. Then, there is a finishing method $\F_h$ such that $\F_h\S_h=I$. If $\S_h=L\S_{-h}$, then $\F_h$ may be chosen to be symmetric; i.e. $\F_h=\F_{-h}L$.
\end{lemma}
\begin{proof} Without loss of generality, we assume that $\S_h$ is represented as in Lemma \ref{lem:Suniq}. Since $\zeta(\emptyset)=1$, there exists an inverse B-series $\zeta^{-1}$, \cite[II.12]{hnw}. We observe that $\F_h:=\zeta^{-1}e_{1}\tr$ is a suitable finishing method. Now suppose $\S_h$ is symmetric, and recall from Subsection \ref{subs:V} that $Le_1=e_1$ and $e_1\tr L=e_1\tr$. Thus, $L\S_{-h}=\S_h$ implies that $ \zeta$ is the same for $\S_h$ and $\S_{-h}$. Hence,  $\F_{-h}L=\zeta^{-1}e_{1}\tr L=\zeta^{-1}e_{1}\tr=\F_h.$
\end{proof}
 
\begin{theorem}\label{theo:evenglob}
Suppose that $\M_h$  is symmetric and of order $p$. Then,  it is of order $p$ relative to a symmetric starting method $\S_h$, with corresponding symmetric finishing method $\F_h$.  Furthermore, the  error for initial data $y_0\in X$ at $x=nh$ is given by
\begin{equation}
E_{nh}y_0-\F_h \M_h^n \S_h y_0=h^{p} c_{p+1}(y_0,\,nh)+h^{p+1} c_{p+2}(y_0,\,nh)+\cdots,\label{eq:globgen}
\end{equation}
where $p$ is even and only even powers of $h$ appear on the right--hand side of \eqref{eq:globgen}.
\end{theorem}
\begin{proof}
The existence of suitable $\S_h$ and $\F_h$ is shown in Theorem \ref{theo:evenloc} and Lemma \ref{lem:fin}.
Let $y_0\in X$ and $x\in\R\setminus\{0\}$ be fixed. Given $n\in \Z\setminus\{0\}$,   define $\mbox{err}(n):=E_x-\F_{ x/n}\M_{ x/n}^{n}\S_{ x/n}.$ Transforming $n\longleftrightarrow -n$, and using the symmetry  of $\M_{h},\S_h$ and $\F_h$, we obtain
 \begin{align*}
\mbox{err}(-n)&=E_x-\F_{- x/n}\M_{- x/n}^{-n}\S_{- x/n}\\
&=E_x-\F_{- x/n}L(L\M_{- x/n}^{-1}L)^n  L\S_{- x/n}\\
&=E_x-\F_{ x/n}\M_{ x/n}^{n}\S_{ x/n}=\mbox{err}(n).
\end{align*}
Thus, $\mbox{err}(n)$ is an even function of $n$. Hence, the expansion
\begin{equation*}
\mbox{err}(n)= ( x/n)^{p} c_{p+1}(y_0,\, x)+( x/n)^{p+2} c_{p+3}(y_0,\, x)+\cdots,
\end{equation*}
may only contain even powers of $n$. Putting  $h=x/n$, we deduce that only even powers of $h$ have non--zero coefficients in \eqref{eq:globgen}.
\end{proof}

\subsection{The underlying one--step method}

Given a method $\M_h$, the map $\Phi_h:X\longrightarrow X$ is an underlying one--step method (UOSM) for $\M_h$ if there is a map $\S_h:X\longrightarrow X^r$ such that
\begin{equation}
\S_h\Phi_h y_0=\M_h\S_h y_0,\qquad y_0\in X.\label{eq:OSM}
\end{equation}
Relation \eqref{eq:OSM} may be represented by a commutative diagram as in Figure
\ref{fig:comm}.

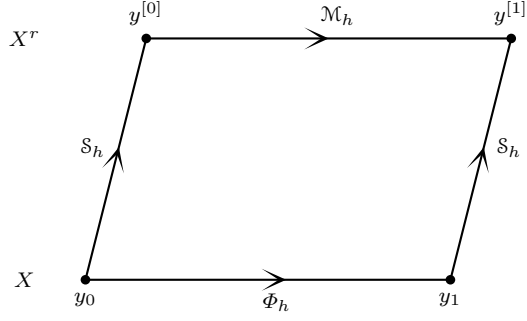
\begin{figure}
\begin{center}
\psset{xunit=8mm,yunit=8mm}
\begin{pspicture}(-2,-0.25)(8,4,5)
\psline{*-*}(0,0)(6,0)
\psline(0,0)(1,4)
\psline(6,0)(7,4)
\psline{*-*}(1,4)(7,4)
\psset{linestyle=none,arrowsize=6pt 3,arrowinset=0.7,arrowlength=1.1}
\psline{->}(0,0)(3.3,0)
\psline{->}(0,0)(0.55,2.2)
\psline{->}(6,0)(6.55,2.2)
\psline{->}(1,4)(4,4)

\uput[-90](0,0){$y_0$}
\uput[-90](6,0){$y_1$}
\uput[90](1,4){$y^{[0]}$}
\uput[90](7,4){$y^{[1]}$}
\uput[-90](3.15,0){$\Phi_h$}
\uput[90](4.15,4){$\M_h$}
\uput[180](0.55,2.2){$\S_h$}
\uput[0](6.55,2.2){$\S_h$}

\rput(-1,0){$X$}
\rput(-1,4){$X^r$}

\end{pspicture}
\end{center}
\caption{Commutative diagram for underlying one-step method}\label{fig:comm}
\end{figure}

The concept of an underlying one--step method  in the linear multistep case is due to Kirchgraber \cite{k86}. The existence  of an underlying one--step method was extended to strictly stable general linear methods and made precise by Stoffer \cite{st93}. For the broader class of zero-stable methods,   the existence and uniqueness  of a formal B-series for $\Phi_h$ and $\S_h$ was
shown in \cite{hlw}.

Because  $\Psi_h=I_{X}+O(h)$,  $\Psi_h:X\longrightarrow X$ is invertible and $\Psi_h^{-1}\Phi_h \Psi_h$ is also a UOSM for $\M_h$:
\begin{equation}
(\S_h \Psi_h)(\Psi_h^{-1}\Phi_h \Psi_h)y_0=\M_h(\S_h \Psi_h) y_0,\qquad y_0\in X.\label{eq:OSMconj}
\end{equation}
This freedom in $\S_h$ and $\Phi_h$ might be restricted in several ways. In \cite{hlw} this is achieved by choosing a finishing method $\F_h:X^r\longrightarrow X$ in advance, and enforcing the  finishing condition
\begin{equation}
\F_h\S_h=I_X.\label{eq:fsi}
\end{equation}
Here, we prefer to specify $\zeta$, the B-series of the first component of $\S_h$. 

Below, we use the notation defined for  Lemma \ref{lem:Suniq}, and define  B-series $\varphi$ and  $[\zeta,\,\xi]$ to represent $\Phi_h$ and $\S_h$ respectively.  Equation \eqref{eq:OSM} now implies the 
tree identities
\begin{align}
\eta(t) &= A(\eta D)(t) +\1 \zeta(t) + \dot U\xi(t),\label{eq:eta3}\\
(\phi\zeta)(t) &=  b\tr (\eta D)(t) + \zeta(t),\label{eq:Ez3}\\
(\phi\xi)(t) &= \dot B(\eta D)(t) + \dot V\xi(t),\label{eq:Ex3}
\end{align}
for a B-series $\eta$ representing the stage values.

\begin{theorem}\label{theo:OSM}
Let $\M_h$ be a consistent zero--stable general linear method, such that the method may be written in the form \eqref{eq:bldiag} with $1$  a simple eigenvalue of $V$. If  $\zeta$ is chosen such that $\zeta(\emptyset)=1$, then there exist unique $\S_h$ and $\Phi_h$ formally satisfying  \eqref{eq:OSM}. 
\end{theorem}
\begin{proof}
For $k=0$, (\ref{eq:eta3}, \ref{eq:Ez3}) and \eqref{eq:Ex3} imply that $\xi(\emptyset)=0$, $\eta(\emptyset)=\1$ and $\varphi(\emptyset)=1$.  For $k\in \N$,  assume that (\ref{eq:eta3}, \ref{eq:Ez3}) and \eqref{eq:Ex3} hold for $|t|\leq k-1$. For $|t|=k$,  $\xi(t)$, $\eta(t)$ and $\phi(t)$ are successively  fixed by the following uniquely soluble rearrangements of (\ref{eq:Ex3}, \ref{eq:eta3}) and \eqref{eq:Ez3}:
\begin{align*}
\xi(t) &= (I-\dot V)^{-1}\big( B (\eta D)(t) + \xi(t) - (\varphi  \xi)(t)     \big),\\
\eta(t) &= A(\eta D)(t) +\1 \zeta(t) + \dot U\xi(t),\\
\varphi(t) &= b\tr(\eta D)(t) +(\varphi(t) +\zeta(t)- (\varphi \zeta)(t)).
\end{align*}
We observe that the terms on the right-hand side of the first and third equations depend only on the given value of $\zeta(t)$ and on  trees of   order less than $k$.  Once $\xi(t)$ is found, the second equation fixes $\eta(t)$. Induction on $k$ now implies the existence of suitable $\xi,\eta$ and $\varphi$. Hence, there exist formal series for $\Phi_h$ and $\S_h$ satisfying identity \eqref{eq:OSM}. 
\end{proof}

\noindent
{\bf Remark:} If $\zeta$ is chosen equal to the first component of the practical starting $\S_h$ found in Lemma \ref{lem:Suniq}, then $E_h$ is a solution of (\ref{eq:Ez3}, \ref{eq:Ex3}) up to $O(h^{p})$.  In that case, we deduce that the corresponding one-step method satisfies
\begin{equation}
E_h y_0-\Phi_h y_0=c_{p+1}(y_0)h^{p+1}+O(h^{p+2}).\label{eq:OSMapprox}
\end{equation}

\begin{corollary} 
Let the  assumptions of Theorem \ref{theo:OSM} hold for a symmetric method $\M_h$, and suppose that
\begin{equation}
\zeta(t)=0,\qquad |t|\;\;\;\mbox{\em odd.}\label{eq:zeven}
\end{equation}
Let $\Phi_h$ and $\S_h=S^{(0)}+hS^{(1)}+\ldots$  denote the corresponding underlying one--step and starting methods. Then, $\Phi_h$ and $\S_h$ are symmetric:
\begin{equation}
 \S_{h}=L\S_{-h},\qquad  \Phi_h= \Phi_{-h}^{-1},\label{eq:OSMsymm}
\end{equation}
where $\Phi_{-h}^{-1}$ denotes a formal inverse, and 
\begin{equation}
\S^{(2q)}=L\S^{(2q)},\quad \S^{(2q+1)}=-L\S^{(2q+1)},\qquad \mbox{\em for }\qquad q=0,\,1,\,2,\,\ldots\label{eq:Ssymm}
\end{equation}
Furthermore, \eqref{eq:OSMapprox} holds with $p$ even.
\end{corollary}
\noindent
\begin{proof}
Let $(\S_h,\,\Phi_h)$ be as in the conclusion of Theorem \ref{theo:OSM}. Let $h$ be replaced by $-h$ in \eqref{eq:OSM} and let $y_0$ be replaced by $\Phi_{-h}^{-1} y_0$. A left--multiplication by $L$ then yields
\begin{equation*}
(L\M_{-h} L)(L\S_{-h})\Phi_{-h}^{-1} y_0=(L\S_{-h})y_0,
\end{equation*}
(where all identities hold as formal B-series). Left--multiplication  by $\M_h$  implies that
\begin{equation*}
(L\S_{-h})\Phi_{-h}^{-1} y_0=\M_{h}(L\S_{-h})y_0.
\end{equation*}
Hence, $(L\S_{-h},\,\Phi_{-h}^{-1})$ also satisfy \eqref{eq:OSM}. Now, by virtue of  \eqref{eq:zeven} and $Le_1=e_1$, the B-series for the first component of $L\S_{-h}$ is equal to $\zeta$, the first component  of $\S_h$.  Thus, Theorem \ref{theo:OSM} implies that $(L\S_{-h},\,\Phi_{-h}^{-1})=(\S_h,\,\Phi_h)$, and we deduce \eqref{eq:OSMsymm}. Identities \eqref{eq:Ssymm} follow from a comparison of the coefficients of $h^{2q}$ and $h^{2q+1}$ in the expansions of $\S_h$ and $L\S_{-h}$.

Substitute $y_{-1}=E_{-h}y_0$ for $y_0$ in \eqref{eq:OSMapprox} and left--multiply by $ \Phi_{-h}$. Then, 
\begin{align*}
 \Phi_{-h}y_0= \Phi_{-h}E_h y_{-1}&= \Phi_{-h}\left( \Phi_h y_{-1}+c_{p+1}(E_{-h}y_0)h^{p+1}+O(h^{p+2})\right)\\
\qquad &=E_{-h}y_{0}+c_{p+1}(y_0)h^{p+1}+O(h^{p+2}),
\end{align*}
where $\Phi_{-h}=I_{X}+O(h)$. Under the transformation $h\longleftrightarrow -h$,  we obtain 
\begin{equation*}
E_h y_0-\Phi_h y_0=(-1)^{p}c_{p+1}(y_0)h^{p+1}+O(h^{p+2}).
\end{equation*}
A comparison with \eqref{eq:OSMapprox} reveals that $p$ is even.
\end{proof}

\section{Examples of symmetric non--parasitic methods}\label{sec:examp}
In this section we  construct a number of symmetric methods,
each of which is consistent and  free of parasitism.  Because we will consider only methods
for which $r=2$ and $V=\diag(1,-1)$, the parasitism there is only a single parasitism  growth factor, equal to $-e_2\tr BU e_2$,
\cite{bhhn}.  Parasitism growth rates are also discussed in \cite{hlw}.
   For convenience, we 
select methods for which $A$ is lower triangular, preferably with
some zero elements on the diagonal.  Many of the methods  
have $r=2$ with $V=\diag(1,-1)$, and some are
G-symplectic.  For this choice of $V$, the two options
$L=\diag(1,1)$ and $L=\diag(1,-1)$ are possible and examples will
be given for each of these.  The terminology $pqrs=4123$ indicates that there are $r=2$ and $s=3$, with order  $p=4$ and stage-order $q=1$. Note that an irreducible method with $rs=22$ can never be free of parasitism because for such a method, $b_{22}=\pm b_{21}$ and $u_{22}=\pm u_{12}$  and hence the $(2,2)$ element of $BU$
equals $2b_{21}u_{12}$ and this can only be zero if the method is reducible.
Hence,  we will start our examples with $rs=23$.

\subsection{Starting and finishing methods}
We will present methods with $r=2$ and $L=\diag(1,\pm 1)$.  For $\pm=+$,
the principal input will be an even function and the second input will
be an odd function.  Suppose the B-series for these are defined by the
coefficient vectors $\xi_1$ and $\xi_2$, where
\begin{equation*}\psset{xunit=6mm,yunit=6mm,dotsize=3pt}
\xi_1(\,\ta\,) = \xi_1(\;\tc\;) = \xi_1(\;\td\;)=
\xi_2(\,\tb\,) = \xi_2(\;\,\te\;\,)=\xi_2(\;\tf\;)= \xi_2)(\;\tg\;)=\xi_2\big(\,\raisebox{-5pt}{\th}\,\big)=0,
\end{equation*}
then it will be sufficient to also specify the required values of $x=[x_1,x_2,x_3,x_4]$,
where
\begin{equation*}\psset{xunit=6mm,yunit=6mm,dotsize=3pt}
x_1=\xi_2(\,\ta\,), \quad x_2= \xi_1(\,\tb\,), \quad x_3 = \xi_2(\;\tc\;), \quad x_4= \xi_2(\;\td\;).
\end{equation*}
Note that the values of $\xi_1(t)$ where $|t|=4$ are irrelevant to the construction of
appropriate starting   values.
Consider the two Runge--Kutta methods
\begin{equation}
\begin{array}{c|c}
c&A\\
\hline
&b\tr
\end{array},\qquad \begin{array}{c|c}
\widehat c&\widehat A\\
\hline
&\widehat b\tr
\end{array}
:=
\begin{array}{c|c}
Pc-\1 b\tr \1&PAP-\1 b\tr P\\
\hline
&-b\tr P
\end{array},
\end{equation}
where $P$ is the stage reversing permutation matrix.  Note that the two Runge--Kutta
methods are exact inverses.  Hence if $R_h$ is the mapping associated with
$(A,b\tr, c)$, then $R_h^{-1}$ will be the mapping associated 
with $(\widehat A,\widehat b\tr, \widehat c)$.

Impose on the $(A,b\tr, c)$ method the order conditions
\begin{align*}
Cb\tr\1 &= x_1,\\
b\tr c &= x_2,\\
Cb\tr c^2 &= x_3,\\
Cb\tr A c &= x_4,
\end{align*}
where $C$ is a constant at our disposal.
Based on $R_h$ and $R_h^{-1}$, we will use a starting method $S_h$, defined by
\begin{align*}
y_1^{[0]} &= \tfrac12(R_hy_0+ R_{-h} y_0),\\
y_2^{[0]} &= \tfrac12 C(R_hy_0- R_{-h} y_0).
\end{align*}
Similarly, we will use a finishing method $F_h$, defined by
\begin{equation}
y_n = \tfrac12 \Big(\widehat R_h\big(y_1^{[n]} + \tfrac1C y_2^{[n]}\big)+ \widehat R_{-h}\big(y_1^{[n]} - \tfrac1C y_2^{[n]}\big)\Big).
\end{equation}
These proposed starting and finishing methods have the property that $F_h\circ S_h = \mbox{id}$ and that
they are consistent with the symmetry of the main method.

Starting methods will be presented in the form of 
\begin{equation}[C, (A,b\tr, c),(\widehat A,\widehat b\tr, \widehat c)]\label{eq:genstart}.
\end{equation}

\subsection{Methods with $rs=23$}\label{SS:23}
Because we will insist on consistent, irreducible, parasitism-free methods,
we will need to reject the case $L=\diag(1,1)$.  The reason for this is that
symmetry would require $b_{22}=0$, $b_{23}=-b_{21}$ and also
$u_{22}=0$, $u_{32}=-u_{12}$.  Hence, the parasitism growth factor
would be $\mu =-2b_{21} u_{12}$, and this would only be zero if either
$b_{21}=0$ or $u_{12}=0$.  However, in each of these cases, the method
reduces to a Runge--Kutta method.  However, methods exist with 
$L=\diag(1,-1)$ and the general case, assuming lower triangular $A$ is
given by 
\begin{equation*}
\left[
\begin{array}{ccc|cc}
a_{11}&0&0&1 & \m u_1\\
a_{21}&a_{22}&0& 1& \m u_2\\
a_{31}&a_{32}&a_{33}& 1 & \m u_1\\
\hline
b_1 & b_2 & b_1 & 1&\m0\\
\beta_1 & \beta_2 &\beta_1 &0 &-1
\end{array}
\right],
\end{equation*}
subject to 
\begin{align*}
2b_1+b_2 &=1,\\
2\beta_1 u_1 + \beta_2 u_2&=0,\\
a_{11} + a_{33}  &= b_1-\beta_1 u_1,\\
a_{21}&= b_1-\beta_1 u_2,\\
2a_{22} &=  b_2-\beta_2 u_2,\\
 a_{31} &= b_1-\beta_1 u_1,\\
a_{32} &= b_2-\beta_2 u_1.
\end{align*}
By consistency, the methods in this family have order $1$ and therefore, by Theorem
\ref{theo:evenglob}, the order is also $2$.  For order $3$, conditions associated
with the trees of that order must be satisfied and, in this case, again using the even order
result, the order must be $4$.

We present three examples of symmetric methods with $rs=23$ and order $4$.
None of these can be G-symplectic because this additional requirement would contradict
the parasitism-free condition.

\subsubsection*{First $4123$ method}
The tableau for the method, which we will name $4123A$, is
\begin{equation*}
\left[
\begin{array}{cc}
A & U\\
B &V
\end{array}
\right]
=
\left[
\begin{array}{ccc|c@{\enspace}r}
\frac16&0&0&1&-\frac13\\
\frac16&\frac16&0&1 &\frac13\\
\frac13&\frac23&\frac16&1 &-\frac13\\
\hline
\frac14&\frac12&\frac14&1 &0\\
\frac14&\frac12&\frac14&0 &-1\\
\end{array}
\right].
\end{equation*}
To verify the order $4$, we need to find a starting method, $y^{[0]} =S_h y_0$ such
that the output after a single step of the method is $y^{[1]} = S_h y(x_0+h) + O(h^5)$.
For this method a suitable choice of the starting values is given by
\begin{equation*}
y_1^{[0]}= y(x_0) + \frac{h^2}{48} y''(x_0),\qquad y_2^{[0]}=  \frac{h}{2} y'(x_0) - \frac{h^3}{32} y'''(x_0).
\end{equation*}
We note that $S_h=LS_{-h}$, as required for a symmetric starting method.
We   need to confirm that the result found by one step of the method is, to within $O(h^5)$,
 equal to 
\begin{align*}
y_1^{[1]}&= y(x_0+h) + \frac{h^2}{48} y''(x_0+h)\\
&=y(x_0) + hy'(x_0) + \frac{25h^2}{48} y''(x_0) + \frac{3h^3}{16} y'''(x_0) + \frac{5h^4}{96} y^{(4)}(x_0) \\
y_2^{[1]}&=  \frac{h}{2} y'(x_0+h) - \frac{h^3}{32} y'''(x_0+h)\\
&=\frac{h}2y'(x_0) + \frac{h^2}2 y''(x_0) + \frac{7h^3}{32} y'''(x_0) + \frac{5h^4}{96} y^{(4)}(x_0). \
\end{align*}
The B--series coefficients for $y^{[0]}_1$ and $y^{[0]}_2$,  corresponding to a tree $t$ are denoted by $\xi_1$ 
and $\xi_2$ respectively, with the target values of the components of $y^{[1]}$ given by the components of
$E\xi$.  These are shown in Table \ref{tbl:4123A} for the empty tree $\emptyset$ and for the $8$ trees
of order up to $4$.  Also shown are the B-series coefficients for the three stages, denoted by $\eta_i$ and the
stage derivatives $(\eta D)_i$, $i=1,2,3$.  Note that the table does not give values for $\eta_i(t)$ where $|t|=4$,
because these are not needed in the evaluation of $\eta D$ up to order $4$.
\begin{table}[h]
\caption{Verification of the order of the method $4123A$}
\label{tbl:4123A}
\begin{equation*}
\psset{xunit=6mm,yunit=6mm,dotsize=2.5pt,linewidth=1.1pt}
\begin{array}{r|cccc@{}c@{}cccc}
t & \emptyset &\ta&\tb&\tc&\td&\te&\tf&\tg&\th\\
\hline
\xi_1(t) &1&0& \frac1{48} &0&0&0&0&0&0\\
\xi_2(t) &0&\frac12& 0&-\frac1{16}\m&-\frac1{32}\m&0&0&0&0\\
\hline
\eta_1(t) & 1&0 & \frac1{48} & \frac1{48}&\frac1{72}\\
(\eta_1D)(t) & 0 & 1 &0&0&\frac1{48}& 0&0&\frac1{48}&\frac1{72}\\
\eta_2(t) &  1&  \frac12 &\frac5{48}&\frac1{48} & \frac1{96}\\
(\eta_2D)(t) & 0 & 1 &\frac12 &\frac14&\frac5{48}&\frac18&\frac5{96}&\frac1{48} & \frac1{96}\\
\eta_3(t) &  1& 1 &\frac{25}{48}&\frac{17}{48}&\frac{25}{144}\\
(\eta_3D)(t) & 0 & 1 & 1&1&\frac{25}{48}&1&\frac{25}{48}&\frac{17}{48}&\frac{25}{144}\\
\hline
(E\eta_1)(t)&1&1&\frac{25}{48}&\frac38&\frac3{16}&\frac5{16}& \frac5{32} &\frac5{48} &\frac{5}{96}  \\
(E\eta_2)(t)&0&\frac12&\frac12&\frac7{16}&\frac7{32}&\frac5{16}&\frac5{32}&\frac5{48}&\frac{5}{96}
\end{array}
\end{equation*}
\end{table}
Practical starting methods can be found 
in the form   \eqref{eq:genstart}
satisfying the order conditions for $x=[\tfrac12,\tfrac1{48},-\tfrac1{16},-\tfrac1{32}]$.
The solution is
\begin{equation}\label{eq:SF4123A}
\left[ -12,\quad
\begin{array}{c|cc}
\frac14 & \frac14 & \m0\\
0 & \frac1{12} & -\frac1{12}\\
\hline
 & \frac1{12} & -\frac1{8}
\end{array},\quad
\begin{array}{c|cc}
\frac1{24} & \frac1{24} & \m0\\
\frac7{24} & \frac1{8} & \m\frac1{6}\\
\hline
 & \frac1{8} & -\frac1{12}
\end{array}\right].
\end{equation}
Here, as for the other methods in this section, one may choose the starting method to be explicit at the price of a more implicit finishing method, as the following alternative starting--finishing combinations indicate:
\begin{equation*}
\left[ \frac{4}{7},\quad
\begin{array}{c|ccc}
0&0 & 0 &0\\
\frac{3}{4}& \frac{3}{4} & 0& 0\\
\frac{8}{7} & \frac{1}{4} & \frac{25}{28}& 0\\
\hline
 & \frac{3}{4} &  \frac{65}{324} &-\frac{49}{648}\m
\end{array},\quad
\begin{array}{c|ccc}
\m\frac{47}{56}&\frac{49}{648} &\m\frac{433}{567}& \m0\\
 -\frac{1}{8}&\frac{49}{648}  &  -\frac{65}{324} & \m0\\
 -\frac{7}{8}&\frac{49}{648} &-\frac{65}{324} &-\frac{3}{4}\\
\hline
 &\frac{49}{648}  & -\frac{65}{324} & -\frac{3}{4}
\end{array}\right],
\end{equation*}
\begin{equation*}
\left[ \frac{2\sqrt{15}}{5},\quad
\begin{array}{c|cc}
\frac{ \sqrt{15}}{12} & \frac{31\sqrt{15}}{180} &-\frac{4\sqrt{15}}{45}\\
 \frac{ 7\sqrt{15}}{12}& \frac{169\sqrt{15}}{720}&-\frac{4\sqrt{15}}{45} \\\hline
  &  \frac{31\sqrt{15}}{180}  & -\frac{4\sqrt{15}}{45}
\end{array},\quad
\begin{array}{c|cc}
0&0 & 0\\
 \frac{\sqrt{15}}{16} &\frac{\sqrt{15}}{16}  & 0 \\
\hline
     &  -\frac{31\sqrt{15}}{180}\m  & \frac{4\sqrt{15}}{45}
\end{array}\right].
\end{equation*}

\subsection*{Second $4123$ method}
The following method, which we will denote as $4123B$, has the advantage
of a zero on the diagonal.
\begin{equation*}
\left[
\begin{array}{cc}
A & U\\
B &V
\end{array}
\right]
=
\left[
\begin{array}{ccc|cr}
\frac14&0&0&1&-\frac16\\
\frac12&0&0&1 &-\frac16\\
\frac12&0&\frac14&1 &-\frac16\\
\hline
\frac13&\frac13&\frac13&1 &0\\
1&-2\m&1&0 &-1\\
\end{array}
\right].
\end{equation*}
An analysis, similar to method 4123A, verifies order 4 with
$x=[0,-\tfrac1{48},\tfrac1{16},\tfrac1{16}]$.
Although starting and finishing methods similar to \eqref{eq:SF4123A}
do not exist, using two stage Runge--Kutta methods, they do exist
with three stages.  A possible triple is:  
\begin{align*}
&\left[ -12, \quad
\begin{array}{c|ccc}
0&0&0&0\\
1& \frac1{24} & \frac{23}{24}&0\\
\frac34 & \frac1{24}&\frac1{24} & \frac23\\
\hline
&\frac1{24}&\frac1{24} & -\frac1{12}\m
\end{array}, \quad
\begin{array}{c|ccc}
\frac34&\frac34&0&0\\
1& \frac1{12} & \frac{11}{12}&0\\
0 &\frac1{12}&-\frac1{24}\m & -\frac1{24}\m\\
\hline
&\frac1{12}&-\frac1{24}\m & -\frac1{24}\m
\end{array}\right].
\end{align*}

\subsection*{A $4223$  method}
The following method, named 4223A,  is found to have stage order $2$,
\begin{equation*}
\left[
\begin{array}{cc}
A & U\\
B &V
\end{array}
\right]
=
\left[
\begin{array}{ccc|cr}
\frac18&0&0&1&-\frac12\\
0&\frac14&0&1 &1\\
\frac14&\frac34&\frac18&1 &-\frac12\\
\hline
\frac16&\frac23&\frac16&1 &0\\
\frac16&\frac16&\frac16&0 &-1\\
\end{array}
\right].
\end{equation*}
Suitable starting values are
\begin{equation*}
y_1^{[0]}= y(x_0) ,\qquad y_2^{[0]}=  \frac{h}{4} y'(x_0) - \frac{h^3}{96} y'''(x_0),
\end{equation*}
corresponding to $x=[\frac14,0,-\tfrac1{48},-\tfrac1{96}]$.
No finishing method is required other than $y_n=y_1^{[n]}$ and the starting
method can be defined by $y_2^{[0]}= \tfrac12(R_h y_0 -R_{-h}y_0)$ where
$R_h$ is the Runge--Kutta method with tableau
\begin{equation*}
\begin{array}{c|cc@{}c}
0 & \\
\frac14 & \frac14\\
\frac14 & 0 & \frac14\\
\hline
& \frac7{12} & -\frac1{6}\m &-\frac1{6}
\end{array}.
\end{equation*}

\subsection*{A special $4123$ method}

The method to be  named 4123C is defined by
\begin{equation*}
\left[
\begin{array}{cc}
A & U\\
B &V
\end{array}
\right]
=
\left[
\begin{array}{ccc|cc@{}}
0&0&0&1&1\\
\frac7{12}&\frac5{12}&0&1&-1\m\\
\frac1{12}&-\frac16\m&\frac1{12}&1 &1\\
\hline
\frac13&\frac13&\frac13&1 &0\\
\frac14&\frac12&\frac14&0 &-1\m\\
\end{array}
\right].
\end{equation*}
This method is interesting because, although it is symmetric,
the diagonal of $A$ is \emph{not} symmetric.

Using $x=[\tfrac12, -\tfrac1{24},-\tfrac18,-\tfrac1{48}]$, a
starting--finishing triple is found:
\begin{equation*}
\left[
12,\quad
\begin{array}{c|cc}
0& 0 & 0\\
\frac14 & \frac5{24} & \frac1{24}\\
\hline
 & \frac5{24} & -\frac1{6}\m
\end{array}, \quad
\begin{array}{c|cc}
\m\frac5{24} & \frac5{24} & \m0\\
-\frac1{24} & \frac1{6} & -\frac5{24}\\
\hline
 & \frac1{6} & -\frac5{24}
\end{array}\right].
\end{equation*}
\subsection{Methods with $rs=24$}
\subsubsection*{First method with $L=\diag(1,-1)$}
We now search for symmetric methods of the form
\begin{equation*}
\left[
\begin{array}{rrrr|cc@{}}
&&&&1 & u_1\\
&\begin{pspicture}(0,0)(0,0)\rput(2.75mm,-2.5mm){$A$}\end{pspicture}&&& 1& u_2\\
&&&& 1 & u_2\\
&&&& 1 & u_1\\
\hline
b_1 & b_2 &b_2 & b_1 & 1&0\\
\beta_1 & \beta_2 &\beta_2&\beta_1 &0 &-1\m
\end{array}
\right],
\end{equation*}
with $\beta_1 u_1+\beta_2 u_2=0$ (to eliminate parasitism) and order $4$.
We give an example which will be named 4124A:
\begin{equation*}
\left[
\begin{array}{cc}
A & U\\
B &V
\end{array}
\right]
=
\left[
\begin{array}{rrrr|rr}
0&0&0&0&1 & \frac16\\
0&\frac14&0&0& 1& \frac16\\
0&\frac12&\frac14&0& 1 & \frac16\\
0&\frac12&\frac12&0& 1 & \frac16\\
\hline
-\frac16 & \frac23 &\frac23 & -\frac16 & 1&0\\
-1 & 1 &1&-1 &0 &-1
\end{array}
\right].
\end{equation*}
This method has the same symmetry, defined by $L=\diag(1,-1)$, as in Subsection \ref{SS:23} and it is possible to use similar starting and finishing methods, An analysis, which will not be included, leads to a starting--finishing pair
defined from $x=[0,-\tfrac1{24}, -\tfrac3{16},-\tfrac1{16}]$.  The triple defining the pair
is
\begin{equation*}
\left[
18, \quad
\begin{array}{c|cc}
0& 0 & 0\\
\frac14 & \frac1{6} & \frac1{12}\\
\hline
 & \frac1{6} & -\frac1{6}\m
\end{array},\quad
\begin{array}{c|cc}
\frac1{4} & \frac1{4} & \m0\\
0 & \frac1{6} & -\frac1{6}\\
\hline
 & \frac1{6} & -\frac16
\end{array}
\right].
\end{equation*}
 
\subsubsection*{G-symplectic method with $L=\diag(1,-1)$}
By imposing the requirements of Theorem \ref{theo:gsym},
a G-symplectic symmetric method can be constructed
with $G=\diag(1,-\frac13)$, $D=\diag(-\frac16,\frac23,\frac23,-\frac16)$
and order $4$.
This method, denoted by 4124B, has the tableau
\begin{equation*}
\left[
\begin{array}{cc}
A & U\\
B &V
\end{array}
\right]
=
\left[
\begin{array}{cccc|cc}
\m\frac16&\m0&\m0&\m0&\m1 & \m1\\
\m\frac1{12}&\m\frac1{12}&\m0&\m0& \m1& \m\frac12\\
\m\frac1{12}&\m\frac16&\m\frac1{12}&\m0& \m1 & \m\frac12\\
\m\frac13&-\frac13&-\frac13& \m\frac16 &\m1& \m1\\
\hline
-\frac16 & \m\frac23 &\m\frac23 & -\frac16 & \m1&\m0\\
-\frac12 & \m1 &\m1&-\frac12 &\m0 &-1
\end{array}
\right].
\end{equation*}
The starting--finishing pair, defined from
$x=[\tfrac12, -\tfrac1{24}, -\tfrac{19}{144},-\tfrac1{36}]$
is given by
\begin{equation*}
\left[
2\sqrt{38},\quad
\begin{array}{c|cc}
0&0&\m0\\
\frac{\sqrt{38}}{24}&\frac{5\sqrt{38}}{152}&\m\frac{\sqrt{38}}{114}\\
\hline
 & \frac{5\sqrt{38}}{152}&-\frac{\sqrt{38}}{38}
\end{array},\quad
\begin{array}{c|cc}
\m\frac{2\sqrt{38}}{57}&\frac{2\sqrt{38}}{57}&\m0\\
-\frac{\sqrt{38}}{152}  & \frac{\sqrt{38}}{38}&-\frac{5\sqrt{38}}{152}\\
\hline
&\frac{\sqrt{38}}{38}&-\frac{\sqrt{38}}{152}
\end{array}
\right].
\end{equation*}

\subsubsection*{Fourth order symmetric methods with $L=\diag(1,1)$}
We will derive symmetric parasitism--free methods with $L=I$ and $a_{11}=0$, based on the assumptions
\begin{align}
b\tr \1 &= 1,\label{eq:4124Ibc}\\
b\tr c^2 &= \tfrac12,\label{eq:4124Ibc2}\\
b\tr A c &=\tfrac16. \label{eq:4124IbAc}
\end{align}
From \eqref{eq:4124Ibc} and  \eqref{eq:4124Ibc2}, it is found that
\begin{equation*}
b_2 =\frac{1}{12 c_2(1-c_2)},\qquad b_1=\frac12-b_2.
\end{equation*}
Without loss of generality, because we can use a diagonal scaling
transformation, assume $u_1=1$ and, to eliminate parasitism,
it follows that $\beta_1=-u_2 \beta_2$.   We will impose the condition
$a_{11}=a_{44}=0$, implying that $\beta_1=b_1$.
From $A+PAP =UV^{-1}B$, and the requirement that $A$ is lower triangular we find that
\begin{equation*}
A = \left[
\begin{array}{cccc}
0 & 0 & 0 &0\\
b_1 (1-u_2)&\frac14+x&0&0\\
b_1 (1+u_2)&b_2-b_1&\frac14-x&0\\
2 b_1&b_2 - \frac{b_1}{u_2}&b_2+\frac{b_1}{u_2}&0
\end{array}
\right],
\end{equation*}
where $x$ is arbitrary.
The value of $u_2$ is determined by the requirement that $A\1=c$ and this gives
\begin{equation*}
u_2 = \frac{-12c_2^3+21c_2^2-9c_2+1}{6c_2^2-6c_2+1} +  \frac{12c_2(c_2-1)x}{6c_2^2-6c_2+1}.
\end{equation*}
For \eqref{eq:4124IbAc} to be satisfied, a complicated condition is obtained.  This is  satisfied for any value of $x$ if and only if $c_2=\frac12$ and this is the value
that will be selected.
We present the matrices defining the method in three cases
$x=-\frac14$, $x=0$ and $x=\frac14$. We denote the corresponding methods as 4124C, 4124D and 4124E:
\begin{equation*}
\left[
\begin{array}{cc}
A & U\\
B &V
\end{array}
\right]
=
\left[
\begin{array}{ccc@{}c|cc@{}}
0&0&0&0&1&1\\
\frac12&0&0&0&1&-2\m\\
-\frac1{6}\m&\frac16&\frac12&0&1&2\\
\frac13&\frac5{12}&\frac14&0&1&-1\m\\
\hline
\frac16&\frac13&\frac13&\frac16&1&0\\
\frac1{6}&\frac1{12}&-\frac1{12}\m&-\frac1{6}\m&0&-1\m
\end{array}
\right],
\end{equation*}
\begin{equation*}
\left[
\begin{array}{cc}
A & U\\
B &V
\end{array}
\right]
=
\left[
\begin{array}{ccc@{}c|cc@{}}
0&0&0&0&1&1\\
\frac14&\frac14&0&0&1&-\frac12\m\\
\frac1{12}&\frac16&\frac14&0&1&\frac12\\
\frac13&\frac2{3}&0&0&1&-1\m\\
\hline
\frac16&\frac13&\frac13&\frac16&1&0\\
\frac1{6}&\frac1{3}&-\frac1{3}\m&-\frac1{6}\m&0&-1\m
\end{array}
\right],
\end{equation*}
\begin{equation*}
\left[
\begin{array}{cc}
A & U\\
B &V
\end{array}
\right]
=
\left[
\begin{array}{cccc|cc@{}}
0&0&0&0&1&1\\
0&\frac12&0&0&1&1\\
\frac1{3}&\frac16&0&0&1&-1\m\\
\frac13&\frac1{6}&\frac12&0&1&-1\m\\
\hline
\frac16&\frac13&\frac13&\frac16&1&0\\
\frac1{6}&-\frac1{6}\m&\frac1{6}&-\frac1{6}\m&0&-1\m
\end{array}
\right].
\end{equation*}
Each of these three methods has order  4 for identical conditions on the starting method.  These are defined by
\begin{equation*}
\psset{xunit=6mm,yunit=6mm,dotsize=2.5pt,linewidth=1.1pt}
\begin{array}{r@{\hspace{5pt}}|@{\hspace{5pt}}c@{\hspace{10pt}}ccc@{\hspace{10pt}}c@{\hspace{10pt}}c@{\hspace{10pt}}c@{\hspace{10pt}}c@{\hspace{10pt}}c}
t & \emptyset &\ta&\tb&\tc&\td&\te&\tf&\tg&\th\\
\hline
\xi_1(t) &1&0&0& 0 &0 &0&0&0&0\\
\xi_2(t) &0&0& -\frac1{12}\m&0&0&\frac1{24}&\frac1{36}&\frac1{48}&\frac1{144}\\
\end{array}.
\end{equation*}
Because $\xi_1$ corresponds to  the identity mapping, the finishing method can be chosen 
as $F_h y^{[0]} = y_1^{[0]}$.

A practical starting method is available in the form $y_1^{[0]} = y_0$ and $y_2^{[0]} = R_h y_0-y_0$,
where $R_h$ is defined by the Runge--Kutta tableau

\begin{equation*}
\begin{array}{c|cccc}
0 &\\
-\frac12\m &-\frac12\m\\
\frac12 & \frac56 & -\frac13\m\\
1 & \frac43 & -\frac56\m & \frac12\\
\hline
& \frac14 & 0 & -\frac13\m & \frac1{12}
\end{array}.
\end{equation*}

\section{Simulations}\label{sec:simul}

We compare the long--time numerical behaviour of several symmetric general linear methods  with that of two symmetric Runge--Kutta methods. One of these RK methods is symplectic, and two of the GLMs are $G$--symplectic. The four low--dimensional Hamiltonian test problems we consider have one or more of the following properties: absence of symmetry, non--separability, chaotic behaviour, or large time derivatives. We compare the efficiency of the methods, as well as their ability   to conserve invariants over long times.

\subsection{The problems}
\paragraph{H\'{e}non--Heiles}
The equations of motion  are defined by the separable Hamiltonian  
\begin{equation*}
H(p,q)=\tfrac{1}{2}(p_{1}^{2}+p_{2}^{2})+\tfrac{1}{2}( q_{1}^{2}+q_{2}^{2})+q_1 q_2^2-\tfrac{1}{3} q_2^3.
\end{equation*}
 The initial conditions are taken \cite{hlw} so that $H=\frac{1}{7}$ :
\begin{equation*}
[p_1,p_2,q_1,q_2]\tr = \Big[\sqrt{\tfrac{152}{875}},0.2,0,0.3\Big]\tr.
\end{equation*}
The solution is chaotic. In the experiments, the time-step $h=0.25$, and the final time $T=10^6$.
\paragraph{Double pendulum}
The equations of motion  are defined by the non--separable Hamiltonian  
\begin{equation*}
H(p,q)=\frac{p_{1}^{2}+2p_{2}^{2}-2p_{1}p_{2}\cos(q_1-q_2)}{2(1+\sin^{2}(q_1-q_2))}-\cos(q_2)-2\cos(q_1).
\end{equation*}
For $y:=[p\tr,q\tr]\tr$, let $f(y):=J^{-1}\nabla H(y)$, where
 \begin{equation*}
J:= \left[\begin{array}{rr} 0& I\\-I&0\end{array}\right] = 
 \left[\begin{array}{rrrr} 
0 & 0& 1 & 0\\
0 & 0& 0 & 1\\
-1 & 0 & 0 & 0\\
0 & -1 & 0 & 0
\end{array}\right].
\end{equation*}
For $R=$ diag$(-1,-1,1,1)$  or diag$(1,1,-1,-1)$, the system is $\rho$-reversible \cite{st88}; i.e.,
\begin{equation*}
 f(Ry)=-Rf(y),\qquad \mbox{ for }\quad y\in \R^4.
\end{equation*} 

The initial conditions are taken to be
\begin{equation*}
[p_1,p_2,q_1,q_2]\tr = \Big[0,0,3.14,-3.1\Big]\tr.
\end{equation*}
The solution is chaotic with large time--derivatives. Here, $h=0.01$ and  $T=10^4$.
\paragraph{Kepler problem}
This describes 
 the motion of a planet revolving around the sun, which is considered to be fixed at the origin. The equations of motion are defined by the separable Hamiltonian,
\begin{equation*}
H(p,q)=\tfrac{1}{2}(p_{1}^{2}+p_{2}^{2})- \frac{1}{\sqrt{q_{1}^{2}+q_{2}^{2}}},
\end{equation*}
where $q=[q_{1},q_{2}]\tr$ are the generalized position coordinates of the body and $p=[p_{1},p_{2}]\tr$ are the generalized momenta.  For $y:=[p\tr,q\tr]\tr$, let $f(y):=J^{-1}\nabla H(y)$.
Then, the system is multiply $\rho$-reversible  for 
\begin{equation*}
R= \mbox{diag}(-1,-1, 1, 1), \;\mbox{diag}(1,1, -1, -1),\;\mbox{diag}(1,-1, -1, 1),\;\;\mbox{ or }\;\;\mbox{diag}(-1,1, 1, -1).
\end{equation*}
 The initial conditions are taken to be
\begin{equation*}
  [p_1,p_2,q_1,q_2]\tr = \Big[0,\sqrt{\tfrac{1+e}{1-e}},1-e,0\Big]\tr,
\end{equation*}
for $e=0.6$. The solution is a closed orbit with moderately large  time derivatives.  The angular momentum error is plotted in addition to the Hamiltonian error. Here, $h=0.01$ and  $T=10^4$.

\paragraph{Transformed Lotka--Volterra}
The equations of motion are defined by the separable Hamiltonian 
\begin{equation*}
H(p,q)=p-\exp(p)+2q-\exp(q).
\end{equation*}
This system lacks any obvious symmetry. The initial conditions are taken to be
\begin{equation*}
[p,q]\tr = \Big[\ln 2,\ln 3\Big]\tr.
\end{equation*}
The solution is a non--symmetric orbit in the positive quadrant $p,\,q>0$. Here, $h=0.1$ and  $T=10^3$.

\subsection{Methods used in the simulations}
The following  methods are competitively compared in the initial simulations:
\begin{itemize}
\item Method 4223 from Subsection 6.2: this is symmetric.
\item Method 4124B from Subsection 6.3: this is  symmeric and $G$--symplectic.
\item  Method 4124D from  Subsection 6.3: this is symmetric. 
\item  The DIRK  4115 method: a $5$--step Suzuki composition of the implicit midpoint 2111 method, see \cite[Chapter II]{hlw}: this is symmetric and symplectic. (This is  more efficient than the familiar $3$--step 4113 DIRK composition method due to far smaller error constants.)
\end{itemize}
Simulations with two other methods serve to interpret  the initial results:
\begin{itemize}
\item The  4113 Lobatto IIIB method, \cite[Chapter XI]{hlw}: this symmetric, but not symplectic.
\item Method 4124P from  \cite{bhhn}: this is symmetric and $G$--symplectic.
\end{itemize}

\subsection{Numerical simulations}
As with long--time Runge--Kutta experiments, we use compensated summation and a tight error tolerance for implicit iterations in an attempt to reduce the effects of rounding error. In order to reduce potential parasitic effects, we also use an accurate starting method for the multivalue experiments.

\paragraph{Timings}

 In Table \ref{tb:time} details of the CPU and stopwatch times for each of the experiments are summarised. 
  \begin{table}
 \begin{center}
 \vspace*{.1in}
 \caption{Timings for specific methods with various problems}\label{tb:time}
\begin{tabular}[htbp]{|@{\hspace{2pt}}c@{\hspace{2pt}}|@{\hspace{2pt}}c@{\hspace{2pt}}|c|c|c|c|}
		\hline
		& 			& H-H 			& DP 			& K 			& TLV 		\\ \hline
		4223	& Stopwatch	& $5.7160\times10^3$ s 	& $1.5971\times10^3$ s	& $1.0489\times10^3$ s	& \phantom{0}8.8033 s 	\\
		& CPU 		& $5.7846\times10^3$ s 	& $1.6153\times10^3$ s	& $1.0629\times10^3$ s	& \phantom{0}8.7517 s	\\ \hline
		4124B   & Stopwatch & $7.4698\times10^3$ s	& $2.1707\times10^3$ s	& $1.4358\times10^3$ s 	& 11.7559	s\\
		& CPU 		& $7.5623\times10^3$ s	& $2.1830\times10^3$ s	& $1.4512\times10^3$ s	& 11.8249	s\\ \hline
		4124D	& Stopwatch & $4.9182\times10^3$ s	& $1.4974\times10^3$ s	& $0.9398\times10^3$	 s	& \phantom{0}7.9902	s\\
		& CPU 		& $4.9746\times10^3$ s	& $1.4908\times10^3$ s	& $0.9512\times10^3$ s	& \phantom{0}7.9405	s\\ \hline
		5-DIRK 	& Stopwatch & $9.3599\times10^3$ s	& $2.5711\times10^3$ s	& $1.5986\times10^3$ s	& 13.1807	s\\
		& CPU 		& $9.9014\times10^3$ s	& $2.7857\times10^3$ s	& $1.6254\times10^3$ s	& 13.8061	s\\
		\hline
	\end{tabular}
\end{center}
\end{table}

\begin{figure}
	\centering
	\subfloat[4223]{
		\includegraphics[width=\textwidth]{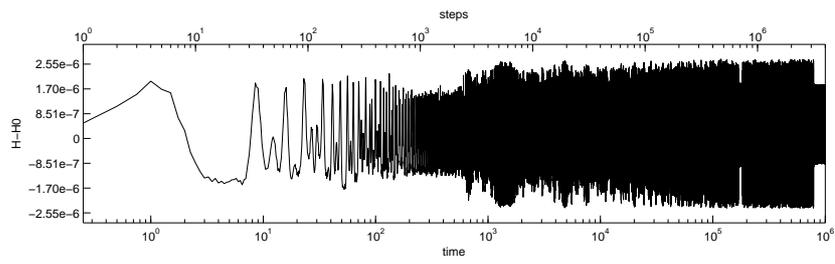}
    }\\
	\subfloat[4124B]{
		\includegraphics[width=\textwidth]{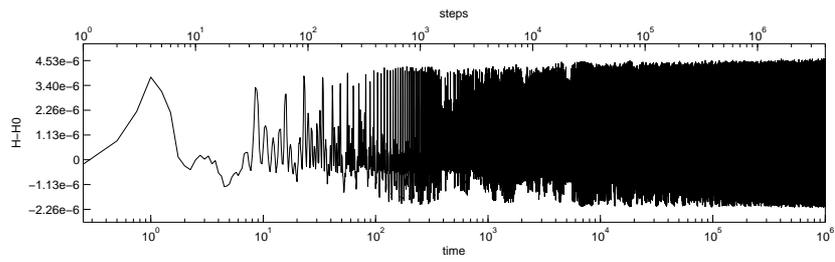}
   }\\
	\subfloat[4124D]{
		\includegraphics[width=\textwidth]{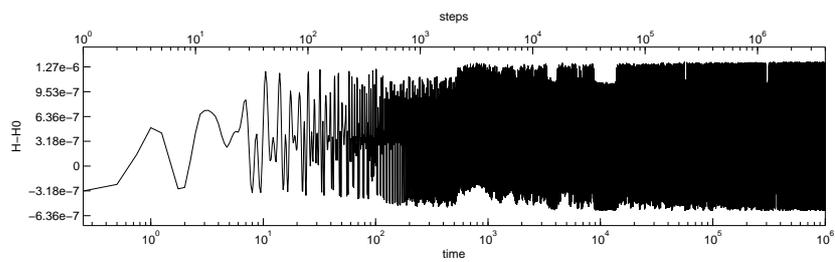}
   }\\
	\subfloat[5-jump DIRK]{
		\includegraphics[width=\textwidth]{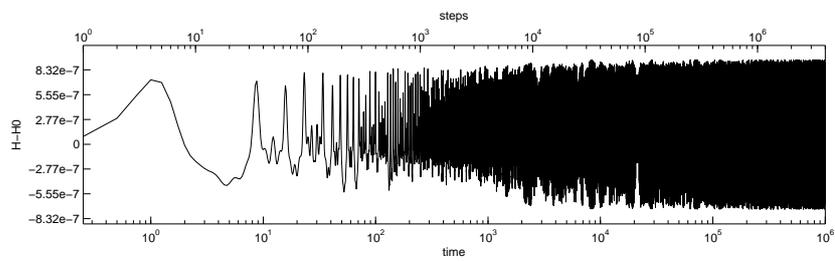}
    }
    \caption{H\'{e}non-Heiles problem}
\end{figure}
\begin{figure}
	\centering
	\subfloat[4223]{
		\includegraphics[width=\textwidth]{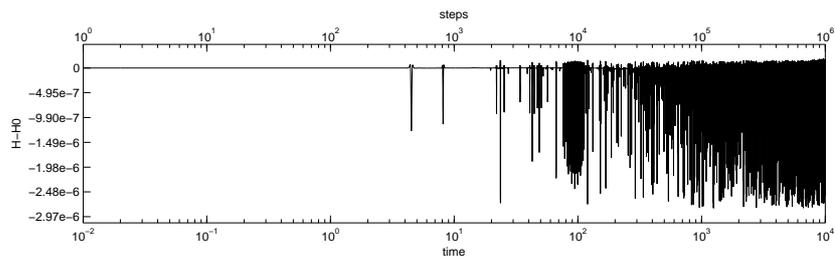}
	}\\
	\subfloat[4124B]{
		\includegraphics[width=\textwidth]{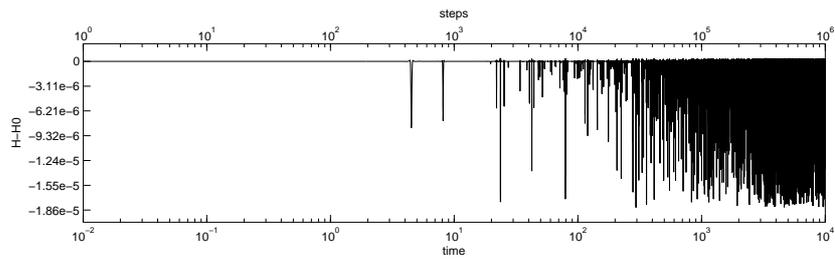}
	}\\
	\subfloat[4124D]{
		\includegraphics[width=\textwidth]{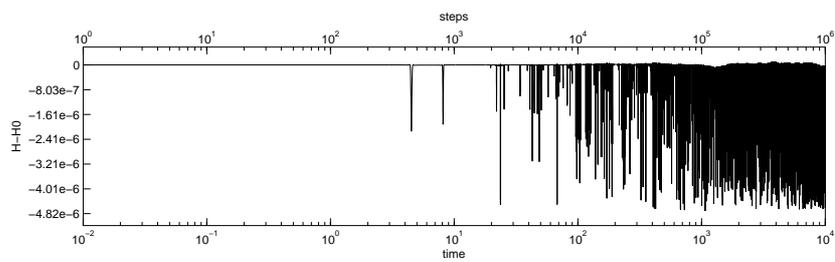}
	}\\
	\subfloat[5-jump DIRK]{
		\includegraphics[width=\textwidth]{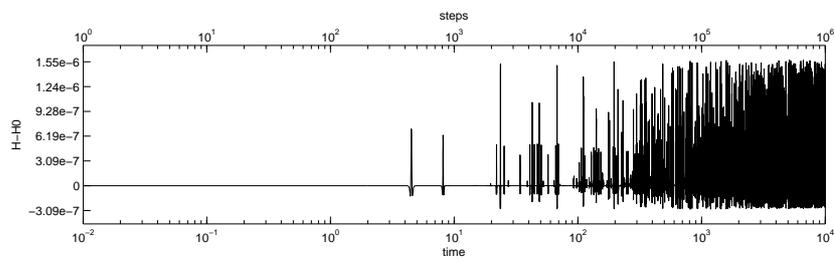}
	}
	\caption{Double pendulum problem}
\end{figure}
\begin{figure}
	\centering
	\subfloat[4223]{
		\includegraphics[width=\textwidth]{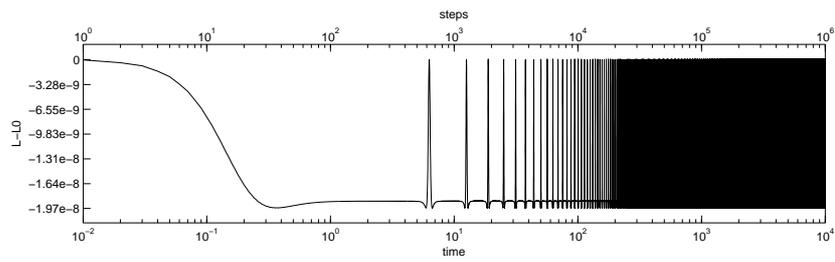}
	}\\
	\subfloat[4124B]{
		\includegraphics[width=\textwidth]{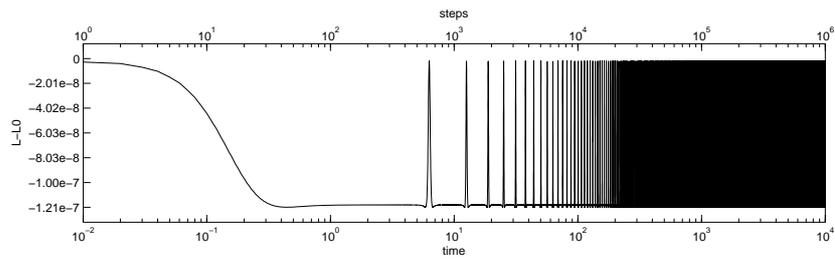}
	}\\
	\subfloat[4124D]{
		\includegraphics[width=\textwidth]{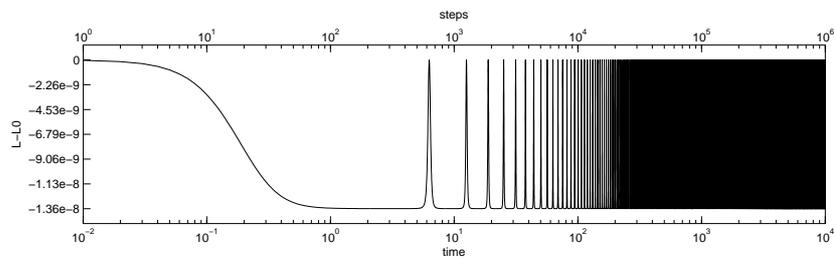}
	}\\
	\subfloat[5-jump DIRK]{
		\includegraphics[width=\textwidth]{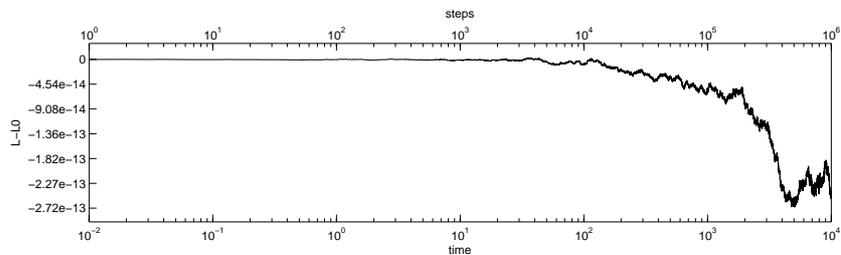}
	}
	\caption{Kepler problem: Angular momentum}
\end{figure}
\begin{figure}
	\centering
	\subfloat[4223]{
		\includegraphics[width=\textwidth]{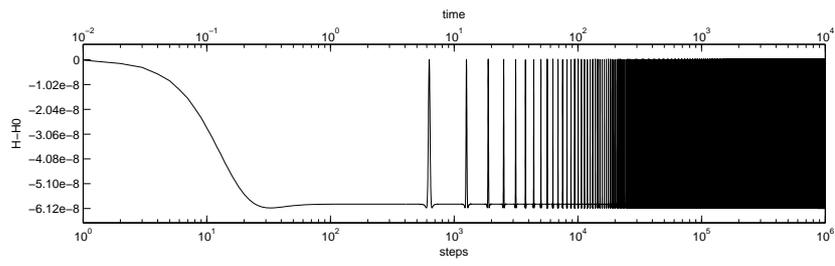}
	}\\
	\subfloat[4124B]{
		\includegraphics[width=\textwidth]{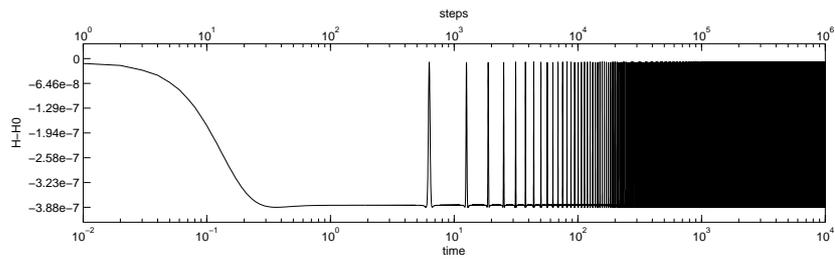}
	}\\
	\subfloat[4124D]{
		\includegraphics[width=\textwidth]{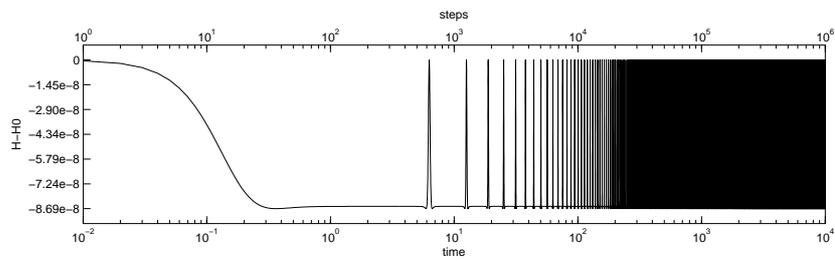}
	}\\
	\subfloat[5-jump DIRK]{
		\includegraphics[width=\textwidth]{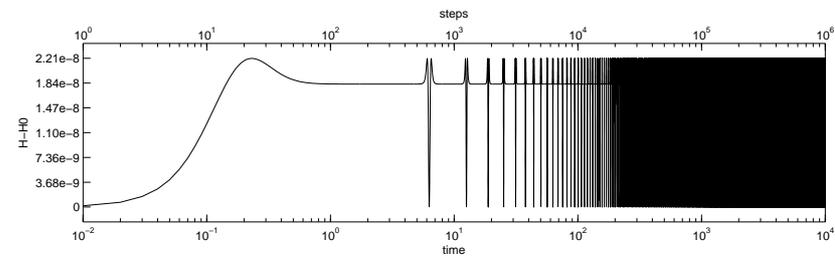}
	}
	\caption{Kepler problem: Hamiltonian}
\end{figure}
\begin{figure}
	\centering
	\subfloat[4223]{
		\includegraphics[width=\textwidth]{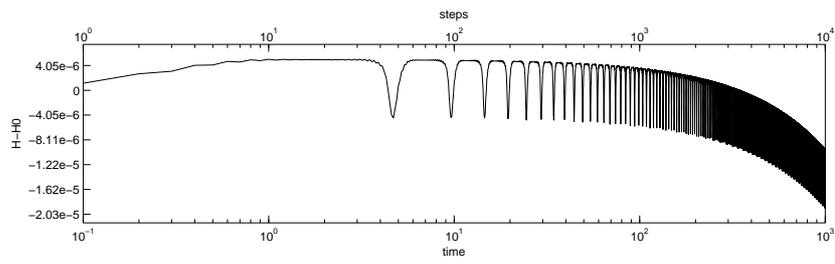}
	}\\
	\subfloat[4124B]{
		\includegraphics[width=\textwidth]{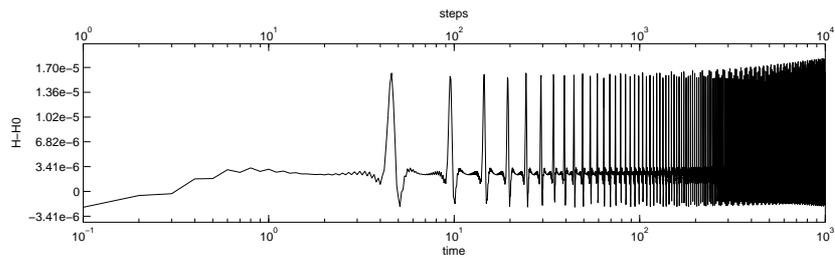}
	}\\
	\subfloat[4124D]{
		\includegraphics[width=\textwidth]{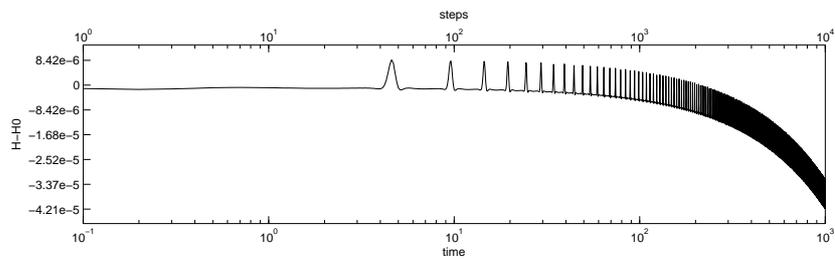}
	}\\
	\subfloat[5-jump DIRK]{
		\includegraphics[width=\textwidth]{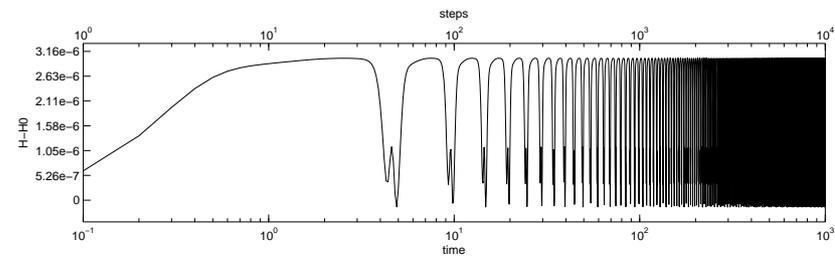}
	}
	\caption{Transformed Lotka-Volterra problem}
\end{figure}

\begin{figure}
	\centering
	\subfloat[H\'{e}non-Heiles]{
		\includegraphics[width=\textwidth]{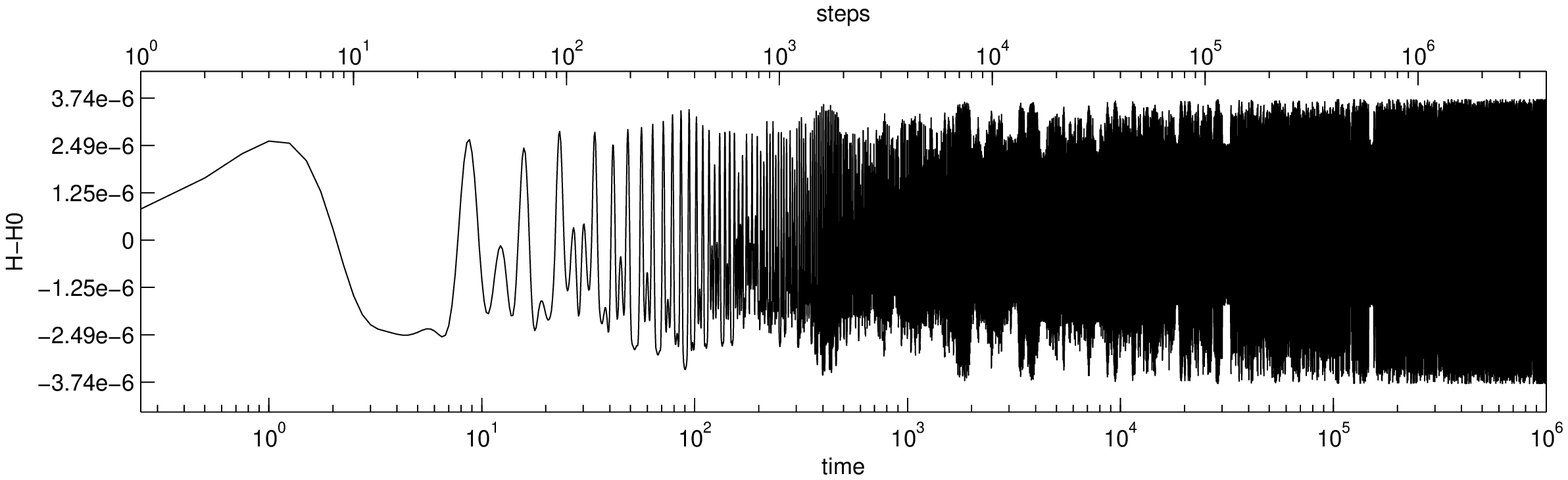}
	}\\
	\subfloat[Double Pendulum]{
		\includegraphics[width=\textwidth]{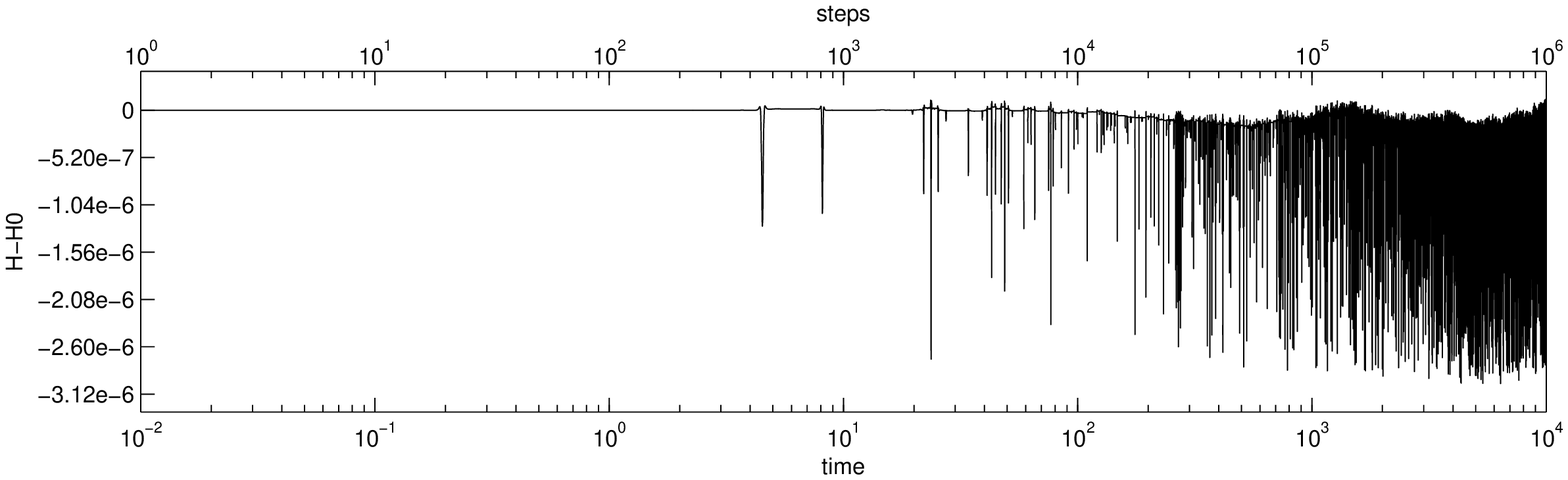}
	}\\
	\subfloat[Transformed Lotka Volterra]{
		\includegraphics[width=\textwidth]{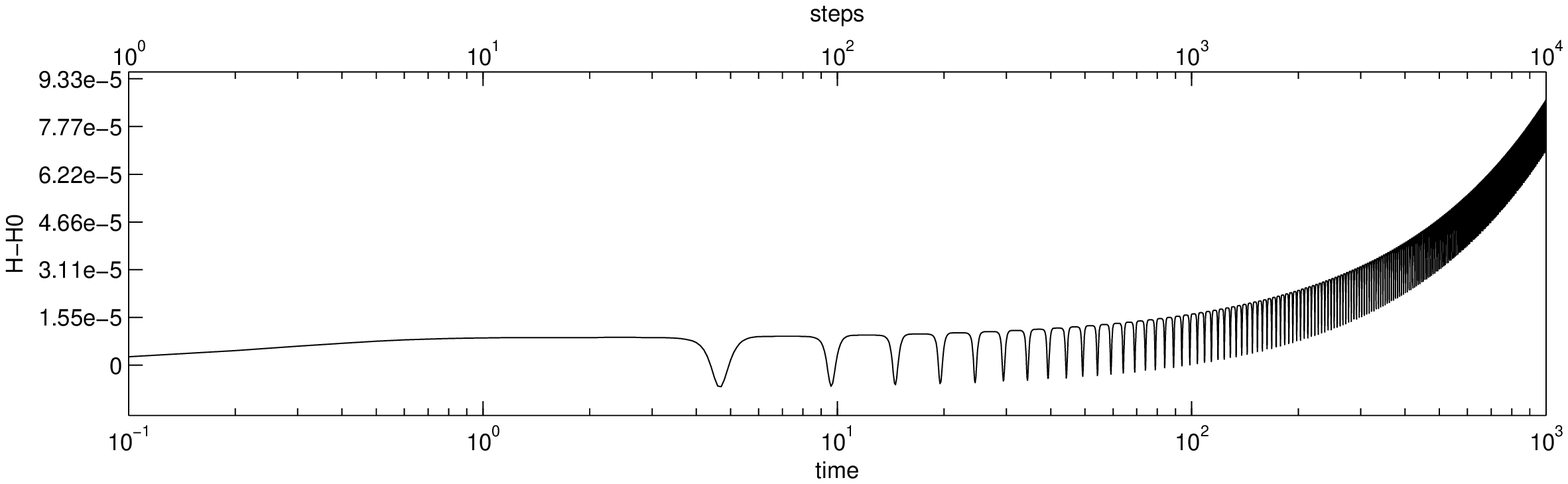}
	}
	\caption{Lobatto IIIB Experiments}
\end{figure}

\begin{figure}
	\centering
	\subfloat[Kepler: Angular Momentum]{
		\includegraphics[width=\textwidth]{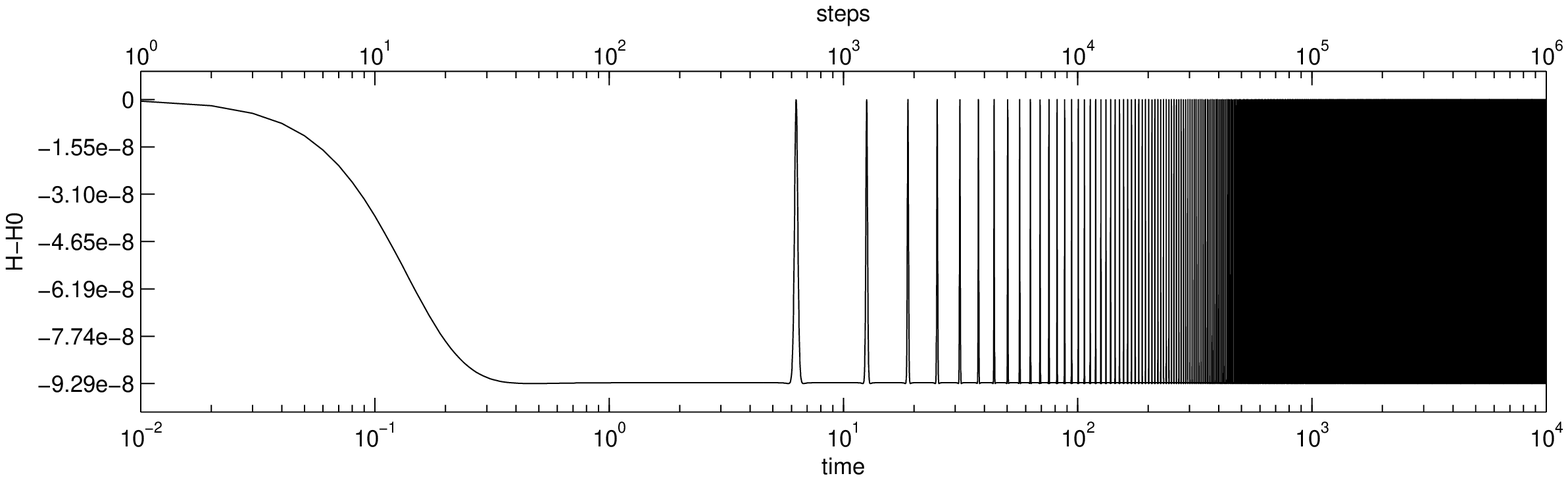}
	}\\
	\subfloat[Kepler: Hamiltonian]{
			\includegraphics[width=\textwidth]{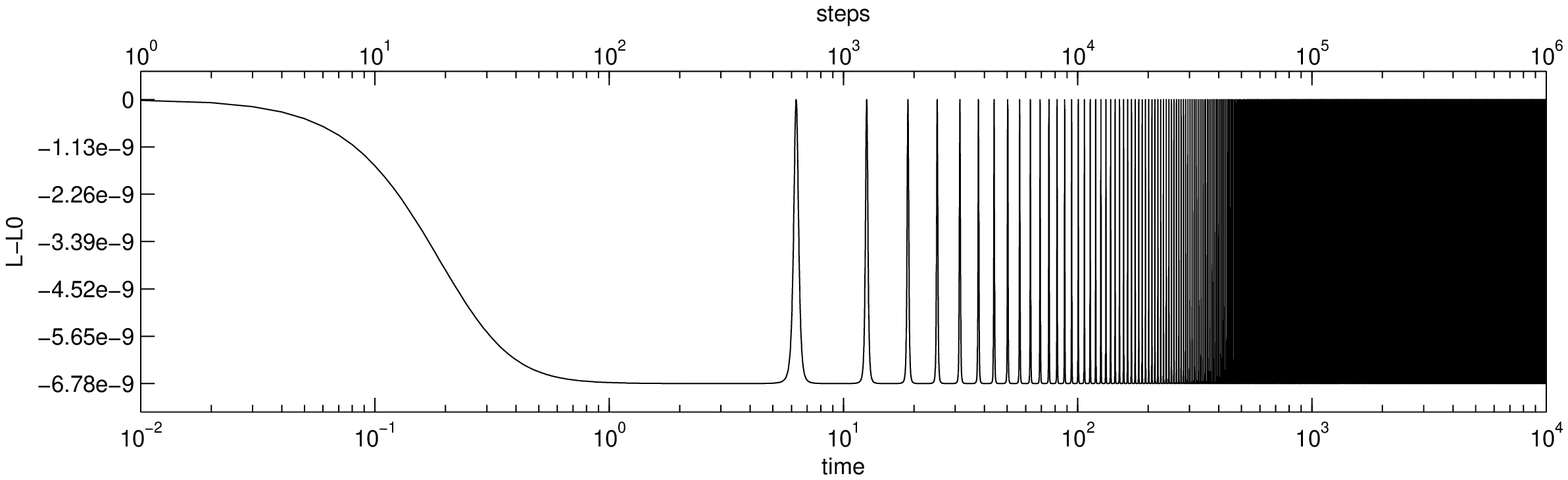}
	}
	\caption{Lobatto IIIB Experiments}
	\includegraphics[width=\textwidth]{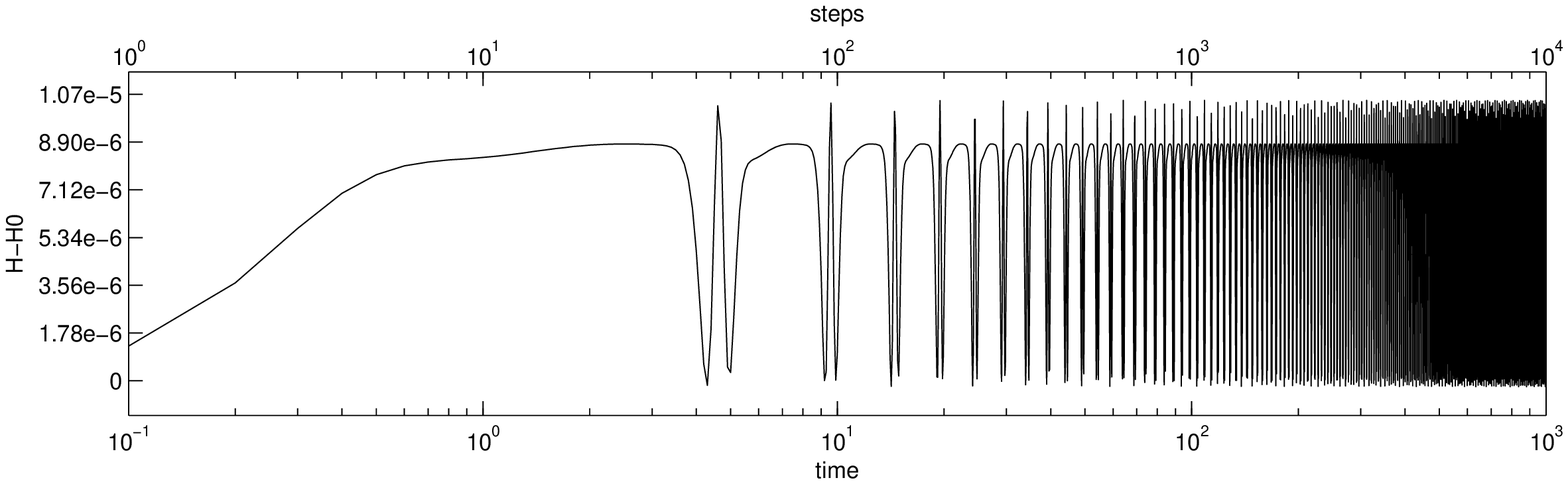}
	\caption{4124P: Transformed Lotka--Volterra}
\end{figure}

\subsection{Interpretation of the simulations} 
Numerical errors in computing the invariants  proceed from several potential sources:\\
(a) The underlying one--step method does not possess the  geometric properties required for the problem.\\
(b) Small periodic deviations occur, for example, when  a conjugate symplectic UOSM approximately conserves a modified Hamiltonian $H_h$, and $H_h$ deviates from the true Hamiltonian  by a small, roughly periodic quantity.\\
(c) Parasitism.

Classically, we think of  the effects of (a) and (c) as being clear--cut. However, the lack of symplecticity in high--order symmetric methods may take a very long time to manifest itself, see e.g. the behaviour of Lobatto IIIA in \cite{fhp}. This is also true of the effects of higher--order parasitism. In order to distinguish computationally the effects due to these two possible causes for the purely symmetric methods 4223 and 4124D, we have also presented results for the 4113 Lobatto IIIB method, which has similar properties to the UOSMs of 4223  and 4124D. Finally, we have also shown results for the symmetric $G$--symplectic 4124P method applied to the TLV problem, as an improvement on those of 4124B.

\paragraph{H\'{e}non--Heiles} All methods exhibit broadly similar conservation behaviour. 
In the absence of parasitism,  it is unsurprising that the results for the $G$--symplectic method 4124B should resemble those of the Suzuki 4115 DIRK. Also, the behaviour of the purely symmetric 4223 and 4124D methods may be explained in terms of their UOSMs, which are closely related to the 4113 Lobatto IIIB method. Following  the explanation of \cite{fhp} for symmetric Runge--Kutta methods,  the fact that $H(p,\,q)$ is a cubic polynomial  implies that the bushy trees in the numerical modified Hamiltonian vanish for order greater than $4$. This permits the existence of an exact modified Hamiltonian  for the UOSM of a symmetric non-symplectic method of order $4$. Thus, even for chaotic solutions, one can expect conservation of a modified Hamiltonian, in the absence of parasitism.
\paragraph{Double pendulum} All methods exhibit broadly similar conservation behaviour.  The system is $\rho$--reversible, but as the behaviour is chaotic, no   analog of the symmetric conservation result, \cite[Theorem XI.3.1]{hlw}, would seem to hold in this case.  Comparing the graphs for  4124D and 4113 Lobatto IIIB, we see broadly similar behaviour. We would therefore attribute any minor deviations in the Hamiltonian as due to properties of the UOSM, rather than to higher--order parasitism.
\paragraph{Kepler} The quadratic angular momentum is exactly conserved by the symplectic Suzuki 4115  DIRK, apart from random round--off errors.  Otherwise, all methods exhibit similar conservation behaviour. Again, in the absence of parasitism, this is what one would expect for the $G$--symplectic 4124B method. The conservation behaviour for 4223 and 4124D follows that of 4113 Lobatto IIIB. In this case, Kepler is both integrable and  reversible. Although the exact hypotheses of  \cite[Theorem XI.3.1]{hlw} are not satisfied here,  the situation is sufficiently similar to conjecture that symmetric UOSMs conserve invariants to   $O(h^p)$ uniformly in time, in the absence of parasitism.
\paragraph{Transformed Lotka--Volterra} This is a Hamiltonian problem without symmetry. In the initial simulations, only the Suzuki 4115 DIRK exhibits satisfactory approximate conservation of the Hamiltonian. The lack of symmetry in the problem and the lack of symplecticity in the UOSMs for 4223 and 4124D methods explains the poor results in those cases. Although 4124B  roughly conserves the Hamiltonian, there is a hint of parasitism at the end of the computation. The results for the $G$--symplectic method 4124P show that good conservation is possible for general linear methods.

\subsection{Conclusions}
All methods performed similarly on the first three problems: H\'{e}non--Heiles, Double Pendulum and Kepler, except that angular momentum was exactly conserved only by the exactly symplectic Runge--Kutta method. Although the errors for the Suzuki 4115 DIRK were about $4$ times smaller than those of 4124D  for the fixed time--steps used, the timings indicate that the latter method is slightly more efficient. Since 4124D only has $2$ implicit stages, one would expect this efficiency advantage to increase for larger problems.

Although parasitism did not develop for these problems, despite chaotic behaviour, large derivatives and long time--intervals, further theoretical work and computational tests would be needed before general linear methods could be applied to other problems with complete confidence. In the absence of parasitism, it appears that symmetric general linear methods behave in the same way as  symmetric Runge--Kutta methods, whllst  $G$--symplectic GLMs  behave similarly to  symplectic RKMs, with the exception that quadratic quantities are not   conserved exactly. In particular, symmetric GLMs are not suitable  for non--symmetric Hamiltonian systems, such as the transformed Lotka--Volterra problem.



\begin{acknowledgements}
JCB was  supported by Marsden Grant AMC1101.
ATH was assisted by LMS grant 41125.
TJTN was supported by a scholarship from EPSRC UK.
\end{acknowledgements}


\end{document}